\documentclass[12pt]{article}
\usepackage{amssymb}
\usepackage{amsmath}
\usepackage{amsthm}
\usepackage{epsfig}
\usepackage{mathrsfs}
\usepackage{graphicx}
\usepackage{ textcomp }
\usepackage{xcolor}
\usepackage{hyperref}
\usepackage{bbold}
\usepackage{fullpage}

\def\R{\mathbb{R}}
\def\NN{\mathbb{N}}

\def\P{{\cal P}} 
\def\E{{\mathbb{E}}}

\def\kvec{\mathbf{k}}
\def\s{\mathbf{s}}

\def\la{\langle}
\def\ra{\rangle}

\def\to{\rightarrow}

\def\eps{\varepsilon}
\def \o{\omega}

\newtheorem{theorem}{Theorem}[section]
\newtheorem{lemma}[theorem]{Lemma}
\newtheorem{proposition}[theorem]{Proposition}
\newtheorem{definition}[theorem]{Definition}

\newtheorem{remark}[theorem]{Remark}
\newtheorem{rk&ex}[theorem]{Remarks \& Examples}
\newtheorem{corollary}[theorem]{Corollary}

\newcommand{\beqar}{\begin{eqnarray*}}
\newcommand{\eeqar}{\end{eqnarray*}}
\newcommand{\lp}{\left(}
\newcommand{\rp}{\right)}

\newcommand{\be}{\begin{equation}}
\newcommand{\ee}{\end{equation}}

\newcommand{\om}{\omega_1+\omega_2-\omega_3}
\newcommand{\K}{K(\omega_1, \omega_2, \omega_3)}
\newcommand{\phiall}{\varphi(\omega_1)\varphi(\omega_2)\varphi(\omega_3)}
\newcommand{\varo}{\varphi(\omega_1)}
\newcommand{\vart}{\varphi(\omega_2)}
\newcommand{\varth}{\varphi(\omega_3)}
\newcommand{\varom}{\varphi(\omega)}
\newcommand{\muo}{\mu(d\omega_1)}
\newcommand{\mut}{\mu(d\omega_2)}
\newcommand{\muth}{\mu(d\omega_3)}
\newcommand{\muom}{\mu(d\omega)}
\newcommand{\mub}{\mu^B_t}
\newcommand{\mubp}{\mu^{B'}_t}

\newcommand{\doo}{(d\omega_1)}
\newcommand{\dotw}{(d\omega_2)}
\newcommand{\doth}{(d\omega_3)}
\newcommand{\lambp}{\lambda^{B'}_t}
\newcommand{\lambb}{\lambda^B_t}

\allowdisplaybreaks[1]

\title{Isotropic Wave  Turbulence with simplified kernels: existence, uniqueness and mean-field limit for a class of instantaneous coagulation-fragmentation processes }

\author{Sara Merino-Aceituno\\
\textit{Imperial College London}\\
\textit{s.merino-aceituno@imperial.ac.uk}}

\begin{document}

\maketitle

\begin{abstract}
The isotropic 4-wave kinetic equation is considered in its weak formulation using model (simplified) homogeneous kernels. Existence and uniqueness of solutions is proven in a particular setting where the kernels have a rate of growth at most linear. We also consider finite stochastic particle systems undergoing instantaneous coagulation-fragmentation phenomena and give conditions  in which this system approximates the solution of the equation (mean-field limit).

\end{abstract}

\tableofcontents

\paragraph{Acknowledgements.}
This work was carried during the author's PhD at the CCA (Cambridge Centre for Analysis) Doctoral Training Centre. This work and author are greatly indebted to James Norris for all his guidance, teaching and corrections.

Thanks to Cl\'ement Mouhot for suggesting to work on the wave kinetic equation and useful discussions. Thanks to Martin Taylor for pointing at a useful report on `The Spatially Homogenous Boltzmann Equation and Particle Dynamics'. Thanks to Amit Einav and Colm Connaughton for useful discussions. 

During this project, the author was supported by the UK Engineering and Physical Sciences Research Council (EPSRC) grant EP/H023348/1 for the University of Cambridge Centre for Doctoral Training, the Cambridge Centre for Analysis.

\section*{Notation}
\begin{eqnarray*}
\R_+ &=&[0,\infty);\\
\mathcal{B} &=& \mbox{space of bounded measurable functions with bounded support};\\
D&=& \{(\o_1, \o_2, \o_3)\in \R_+^3 \, |\, \o_1 + \o_2 \geq\o_3\};\\
\kvec &&\mbox{wavevector, it belongs to } \R^N;\\
\o(\kvec) && \mbox{dispersion relation};\\
\overline{T} &=& \overline{T}(\kvec_1,\kvec_2,\kvec_3,\kvec) \mbox{ interaction coefficient};\\
\mathcal{P}(\R_+) && \mbox{space of probability measures in }\R_+ \\
\mathcal{M}(\R_+) &&\mbox{set of finite measures on $\R_+$}.
\end{eqnarray*}

\section{Introduction}

Wave turbulence  (\cite{zakharov2004one,zakharov1992kolmogorov,nazarenko2011wave}, \cite[Entry turbulence]{scott2006encyclopedia}) describes weakly non-linear systems of dispersive waves. The present work focuses in the case of 4 interacting waves. 

We start with a brief presentation of the general 4-wave kinetic equation and move quickly to consider the isotropic case with simplified kernels, which is the object of study of the present work, and present the main results.

We give a brief account on the theory of wave turbulence in Section \ref{sec:brief_account}. The rest of the text consists on the proofs of the main theorems.

\subsection{The 4-wave kinetic equation.}

Using in shorthand $n_i=n(\kvec_i,t)$, $n_k=n(\kvec,t)$, $\omega_i=\omega(\kvec_i)$ and $\omega=\omega(\kvec)$, the \textbf{4-wave kinetic equation} is given by
\begin{eqnarray} \label{eq:kinetic_wave_equation}
\frac{d}{dt}n(\kvec, t) &=& 4\pi \int_{\R^{3N}} \overline{T}^2(\kvec_1, \kvec_2, \kvec_3, \kvec) (n_1n_2n_3 +n_1n_2n_k -n_1n_3n_k-n_2n_3n_k) \\
&& \qquad \times \delta(\omega_1+\omega_2-\omega_3 -\omega) \delta(\kvec_1+\kvec_2-\kvec_3-\kvec) d\kvec_1d\kvec_2d\kvec_3. \nonumber
\end{eqnarray}
where $\kvec\in \R^N$ is called \textbf{wavevector}; the function $n=n(\kvec,t)$ can be interpreted as the spectral density (in $\kvec$-space) of a wave field and it is called \textbf{energy spectrum}; $\omega(\kvec)$ is the \textbf{dispersion relation}; and 
$$\overline T_{123k}:= \overline T(\kvec_1, \kvec_2, \kvec_3, \kvec) $$
is the \textbf{interaction coefficient}.

$$E=\int_{\R^N}\o(\kvec)\, n(\kvec)d\kvec, \quad
W=\int_{\R^N} n(\kvec) d\kvec$$
correspond to the total energy and the waveaction (total number of waves), respectively. These two quantities are conserved formally.

\paragraph{Properties of the dispersion relation and the interaction coefficient.} $\omega(\kvec)$ and $T_{123k}$ are homogeneous, i.e., for some $\alpha>0$ and $\beta\in\mathbb{R}$
$$\omega(\xi \kvec) = \xi^{\alpha} \omega(\kvec), \qquad \overline T(\xi \kvec_1, \xi \kvec_2, \xi \kvec_3, \xi \kvec) = \xi^\beta  \overline T(\kvec_1, \kvec_2, \kvec_3, \kvec) \qquad \xi>0.$$

Moreover the interaction coefficient possesses the following symmetries
$$\overline T_{123k}=\overline T_{213k}=\overline T_{12k3}=\overline T_{3k12}.$$

\paragraph{Example: shallow water.} 
In the case of shallow water we deal with weakly-nonlinear waves on the surface of an ideal fluid in an infinite basin of constant depth $h$ small. In this case (\cite{zakharov1999statistical}) we have that $\alpha=1$, $\beta=2$, dimension is 2 and
\begin{equation}\label{eq:T_shallow_water}
T(\kvec_1,\kvec_2,\kvec_3,\kvec)=-\frac{1}{16\pi^2h}\frac{1}{(k_1k_2k_3k)^{1/2}} \left[(\kvec_1\cdot\kvec_2)(\kvec_3\cdot\kvec)+(\kvec_1\cdot \kvec_3)(\kvec_2\cdot \kvec) + (\kvec_1\cdot\kvec)(\kvec_2\cdot\kvec_3)\right].
\end{equation}
In general $T$ will be given by very complex expressions, see for example \cite{zakharov1992kolmogorov}.

\paragraph{Resonant conditions and the $\delta$ distributions.} The delta distributions appearing in  equation \eqref{eq:kinetic_wave_equation} correspond to the so-called resonant conditions:
\begin{eqnarray*}
\kvec_1+\kvec_2&=&\kvec_3+\kvec\\
\o(\kvec_1) + \o(\kvec_2)&=&\o(\kvec_3)+\o(\kvec).
\end{eqnarray*}
This imposes the conservation of energy and momentum in the wave interactions. 

\subsection{The \textit{simplified} weak isotropic 4-wave kinetic equation.}
We focus our study on the weak formulation of the isotropic 4-wave kinetic equation defined against functions in $\mathcal{B}(\R^N)$; the set of bounded measurable functions with bounded support in $\R^N$.  

More specifically, we assume that $n(\kvec)=n(k)$ is a radial function (isotropic). 
Then, using the relation $\o(\kvec) =k^\alpha$, we study the evolution of the \textbf{angle-averaged frequency spectrum} $\mu=\mu(d\o)$ which corresponds to 
 $$\mu(d\omega):=\frac{|S^{N-1}|}{\alpha} \o^{\frac{N-\alpha}{\alpha}}n(\o^{1/\alpha})d\o,$$
 where $S^{N-1}$ is the $N$ dimensional sphere.
 The total number of waves (waveaction) and the total energy are now expressed respectively as
\begin{eqnarray} \label{eq:totalNumberOfWaves}
W&=& \int^\infty_0 \mu(d\omega)\\
E &=& \int^\infty_0 \omega \mu(d\omega). \label{eq:total_energy}
\end{eqnarray}

The weak form of the isotropic equation is given formally by 
\begin{equation} \label{eq:isotropic_4_wave_equation_weak}
\mu_t = \mu_0 + \int^t_0 Q(\mu_s, \mu_s, \mu_s) \, ds
\end{equation}
where $Q$ is defined against functions $f\in \mathcal{B}(\R_+)$ as
\begin{eqnarray*}
\la f, Q(\mu, \mu,\mu) \ra &=& \frac{1}{2} \int_{D}  \mu\doo \mu\dotw \mu\doth K(\o_1,\o_2,\o_3)  \\
&&\qquad \times[ f(\o_1+\o_2-\o_3) + f(\o_3) -f(\o_2) -f(\o_1) ]
\end{eqnarray*}
where $D:= \{ \R^3_+ \cap (\o_1 + \o_2 \geq \o_3)\}$. See appendix \ref{sec:weak_isotropic_wave_eq} for the formal derivation of this equation.

Formally $K=\K$ is written as
\begin{eqnarray} 
\K &=& \frac{8\pi}{\alpha |S^{N-1}|^4}(\om)^{\frac{N-\alpha}{\alpha}}\\
&& \qquad \int_{\lp S^{N-1}\rp^4} d\s_1d\s_2d\s_3d\s\, \overline{T}^2(\o_1^{1/\alpha} \s_1, \o_2^{1/\alpha}\s_2, \o_3^{1/\alpha}\s_3, (\om)^{1/\alpha} \s) \nonumber\\
&& \qquad \qquad \times \delta(\o_1^{1/\alpha}\s_1+ \o_2^{1/\alpha}\s_2- \o_3^{1/\alpha}\s_3- (\om)^{1/\alpha} \s). \nonumber
\end{eqnarray}
 Notice that formally $K$ is homogeneous of degree
\begin{equation}
	\label{eq:lambda} \lambda:=\frac{2\beta-\alpha}{\alpha}.
\end{equation}

\bigskip

\textbf{ Our starting point is equation \eqref{eq:isotropic_4_wave_equation_weak}  considering \textit{simplified kernels} $K$}. 
In this  work we do not study the relation between the interaction coefficient $\overline{T}$ and $K$. 
 Specifically, we will consider the following type of kernels:
\begin{definition} We say that $K$ is a \textbf{model kernel} if
\begin{itemize}
	\item $K: \R_+^3 \to \R_+$;
	\item $K$ is continuous in $\R^3_+=[0,\infty)^3$;
	\item $K$ is homogeneous of degree $\lambda$;
	\item $K(\o_1,\o_2,\o_3)=K(\o_2,\o_1,\o_3)$ for all $(\o_1, \o_2,\o_3)\in \R_+^3$.
\end{itemize}

\end{definition} 
 
 Some examples of model kernels are:
\begin{eqnarray} \nonumber
K(\o_1,\o_2,\o_3) &=& \frac{1}{2}\lp \o_1^p \o_2^q \o_3^r + \o_1^q \o_2^p \o_3^r \rp \quad \mbox{with }p+q+r=\lambda,\\
\K &=& (\o_1\o_2\o_3)^{\lambda/3}, \label{eq:bestK}\\
\K &=& \frac{1}{3}(\o_1^\lambda +\o_2^\lambda + \o_3^\lambda). \nonumber
\end{eqnarray}

The main question we want to address is:
\begin{quote}
\textsc{For which types of kernels $K$ there is existence and uniqueness of solutions for equation \eqref{eq:isotropic_4_wave_equation_weak} and, moreover, can this solution(s) be taken as the mean-field limit of a specific stochastic particle system?}
\end{quote}

The present work gives a positive answer for a particular class of kernels as explained in the next section, but first, for the motivation of the problem, we need to answer the two following questions:

\paragraph{a) Why is it relevant to study the weak isotropic 4-wave kinetic equation with simplified kernels?}
\mbox{}

The present work is inspired on the article \cite{connaughton2009numerical} from the physics literature on wave turbulence. In \cite{connaughton2009numerical} the author works with the 3-wave kinetic equation and considers its isotropic version also assuming simplified kernels. The idea is that the 3-wave kinetic equation can be interpreted as a process where particles coagulate and fragment. This interpretation allows to use numerical methods coming from the theory of coagulation-fragmentation processes, which can be applied to this type of simplified kernels.

As in \cite{connaughton2009numerical}, ignoring the specific shape of the interaction coefficient $\overline{T}$ is not uncommon in the wave turbulence literature; in general the shape of $\overline{T}$ is too complex, too messy to extract information. Moreover, the most important feature in wave turbulence, the steady states called KZ-spectrum, depend only on the parameters $\alpha$, $\beta$ and $N$. That is why in the physics literature $\overline{T}$ plays a secondary role, sometimes no role at all.

It is believed that only the asymptotic scaling properties of the kernel will affect the asymptotic behaviour of the solution. This is similar to what happens in the case of the Smoluchowski's coagulation equation, where homogeneous kernels give rise to self-similar solutions (scaling solutions) in some cases. The hypothesis that solutions become self-similar in the long run under the presence of an homogeneous kernel is called \textbf{dynamical scaling hypothesis}, see \cite{mischler2011uniqueness} for more on this. In the case of wave turbulence we expect this self-similar solutions to correspond to the steady states given by the KZ-spectrum.

Proving the dynamical scaling hypothesis for the simplified isotropic 4-wave kinetic equation under the assumptions of Theorem \ref{eq:existence_solutions_kinetic_4wave} (existence of solutions) will imply proving the validity of the KZ-spectrum for this simplified kernels (if there is correspondence between the two). This would provide a great indication of the mathematical validity of the theory of wave turbulence.

\paragraph{b) Why consider the isotropic case?}

There are examples in the physics literature where the phenomena are considered to behave isotropically (like in Langmuir waves for isotropic plasmas and shallow water with flat bottom).

The main reason though to consider the isotropic case is that it makes easier to get a mean-field limit from discrete stochastic particle systems. Suppose that we want to find a discrete particle system that approximates the dynamics of \eqref{eq:kinetic_wave_equation}. For given waves with wavenumbers $\kvec_1, \kvec_2,\kvec_3$, we want to see if they interact. On one hand, due to the resonance conditions $\kvec$ defined as
$$\kvec= \kvec_1+\kvec_2 -\kvec_3$$
is uniquely determined. On the other hand, on top we must add the constraint
$$\omega=\omega_1+\o_2-\o_3$$
and this in general will not be satisfied. Therefore, if we consider systems with a finite number of particles, in general, interactions will not occur and the dynamics will be constant.

We go around this problem by considering the isotropic case. By assuming that  $n=n(k)$ is a rotationally invariant function, we add the degree of freedom that we need.

\subsubsection{Summary of results}

Next we summarise the main results in the present work. These results are the analogous ones presented in the papers \cite{norris1999smoluchowski,norris2000cluster} for the Smoluchowski equation (coagulation model).

\begin{remark}[Strategy] \label{rem:strategy}
We will adapt the proofs by Norris in \cite{norris1999smoluchowski} and \cite{norris2000cluster} for coagulation phenomena.
In the proof by Norris in \cite{norris1999smoluchowski}  sublinear functions $\varphi:\R_+ \to\R_+$ are used, i.e.,
\begin{eqnarray*}
\varphi(\lambda x) &\leq& \lambda \varphi(x), \quad \lambda \geq 1\\
\varphi(x+y)&\leq& \varphi(x)+\varphi(y).
\end{eqnarray*}
These functions are the key to get bounds because of the following property: let $(\mu^n_t)_{t\geq0}$ be a stochastic coagulation process with $n$ particles,  if initially
$$\langle \varphi, \mu^n_0 \rangle \leq \Lambda$$
for some $\Lambda<\infty$, for all $n\in \NN$, then 
$$\langle \varphi, \mu^n_t \rangle \leq \Lambda \mbox{ for all } n,t.$$
Actually, what we obtain is that 
$$\langle\varphi, \mu^n_t \rangle \leq \langle \varphi, \mu^n_0 \rangle$$
thanks to the sublinearity of $\varphi$;  say that two particles of masses $x, y\in \R_+$ coagulate creating a particle of mass $x+y$, then
\begin{equation} \label{eq:ex_varphi}
\varphi(x+y) \leq \varphi(x) + \varphi(y)
\end{equation}
by sublinearity.

\medskip
In general, this idea to get bounds cannot be applied to the type of stochastic particle processes that we are going to consider because they also include fragmentation phenomena;  we will have that in an interaction two particles of masses $\o_1, \o_2\in\R_+$ disappear and two particles of masses $\om, \o_3\in \R_+$ are created. 

To get bounds on this stochastic process using the method above we need an expression analogous to 
\eqref{eq:ex_varphi}, i.e.,
$$\varphi(\om)+\varphi(\o_3) \leq \varphi(\o_1)+\varphi(\o_2).$$
 Therefore we can use Norris method with the appropriate adaptations for the particular case where $\varphi(\o)=\o+c$ for a constant $c$, which we will take to be one. 
 
 Notice that this works as a consequence of the conservation of the energy (given by the $\o$'s, see \eqref{eq:total_energy}) and the conservation of the total number of particles at each interaction.
\end{remark}

\medskip

\begin{definition}
Consider $\varphi(\o)=\o+1$. We say that a kernel $K$ is \textbf{sub-multiplicative} if 
\begin{equation} \label{eq:bounds_on_K}
\K \leq \varphi(\o_1)\varphi(\o_2)\varphi(\o_3).
\end{equation}
\end{definition}

\paragraph{A. Existence and uniqueness of solutions.}

\begin{definition}[Solution and types of solutions]
\label{def:solutions}
We will say that $(\mu_t)_{t<T}$ is a \textit{local solution} if it satisfies \eqref{eq:isotropic_4_wave_equation_weak} for all bounded measurable functions $f$ of bounded support and such that $\la \o, \mu_t \ra \leq \la \o, \mu_0\ra$
 for all $t<T$. If $T=+\infty$ then we have a \textit{solution}. If moreover,
$$\int^\infty_0\o \mu_t(d\o)$$
is finite and constant, then we say that $(\mu_t)_{t<T}$ is \textit{conservative}.

We call any local solution $(\mu_t)_{t<T}$ such that
$$\int^t_0\la \varphi^2, \mu_s\ra\, ds<\infty\quad \mbox{for all } t<T$$
a \textit{strong solution}.
\end{definition}

\begin{remark}
Observe that we  consider the possibility of having not conservative solutions, implying loss of mass. This will correspond to gelation in coagulation and the concept of finite capacity cascades in Wave Turbulence (see \cite{connaughton2009numerical}).
\end{remark}

\begin{theorem}[Existence and uniqueness of solutions]
\label{eq:existence_solutions_kinetic_4wave} 
Consider equation \eqref{eq:isotropic_4_wave_equation_weak} and a given $\mu_0$ measure in $\R_+$.
Define $\varphi(\o)=\o+1$ and assume that $K$ is submultiplicative model kernel.
Assume further that $\la \varphi, \mu_0 \ra <\infty$ (i.e., initially the total number of waves \eqref{eq:totalNumberOfWaves} and the total energy \eqref{eq:total_energy} are finite). Then, if $(\mu_t)_{t<T}$ and $(\nu_t)_{t<T}$ are local solutions, starting from $\mu_0$, and if $(\nu_t)_{t<T}$ is strong, then $\mu_t= \nu_t$ for all $t<T$. Moreover, any strong solution is conservative.

Also, if $\la \varphi^2, \mu_0\ra<\infty$, then there exists a unique maximal strong solution $(\mu_t)_{t<\zeta(\mu_0)}$ with $\zeta(\mu_0)= \la \varphi^2, \mu_0 \ra^{-1} \la \varphi, \mu_0 \ra^{-1}$.
\end{theorem}

The proof of this theorem will be an adaptation of \cite[Theorem 2.1]{norris1999smoluchowski}.

\paragraph{B. Mean-field limit (coagulation-fragmentation phenomena).}

We will consider a system of stochastic particles undergoing coagulation-fragmentation phenomena. The basic idea is that three particles $(\o_1,\o_2,\o_3)$ with $\o_1+\o_2\geq \o_3$ will interact at a given rate $\K$. In the interaction, first $\o_1$ and $\o_2$ coagulate to form $\o_1+\o_2$ and then, under the presence of $\o_3$ the coagulant splits into two other components which are $\om$ and a new $\o_3$ (fragmentation). So interactions are 
$$[\o_1,\o_2,\o_3] \mapsto [\o_1+\o_2-\o_3,\o_3,\o_3].$$
Note that we assume that $K$ is symmetric in the first two variables because in the interactions the role of $\o_1$ and $\o_2$ is symmetric.

We will define and build for each $n\geq 1$, $(X_t^n)_{t\geq 0}$ a instantaneous coagulation-fragmentation stochastic particle system of $n$ particles (Section \ref{sec:coagulation_fragmentation_stochastic_process}) following the previous ideas.

We will approximate the solutions to the isotropic 4-wave kinetic equation using this coagulation-fragmentation phenomena.  We present here two mean-field limits each of them requiring a different set of assumptions:

\begin{theorem}[First mean-field limit] \label{th:mean_field_limit_unbounded_kernel}
Assume that for $\tilde \varphi(\o) =\o^{1-\gamma}$, $\gamma\in (0,1)$ it holds that $K$ is a model kernel with
$$\K \leq \tilde\varphi(\o_1)\tilde\varphi(\o_2)\tilde\varphi(\o_3).$$
Assume also that $\la \o, X_0^n \ra$ is bounded uniformly in $n$ by $\la \o, \mu_0\ra <\infty$, and
$$X^n_0\to \mu_0 \quad \mbox{weakly.}$$
Then the sequence of laws $(X^n_t)_{n\in\NN}$ is tight in the Skorokhod topology. Moreover, under any weak limit law, $(\mu_t)_{t\geq 0}$ is almost surely a solution of equation \eqref{eq:isotropic_4_wave_equation_weak}.
In particular, this equation has at least one solution. 
\end{theorem}

The proof of this theorem will be an adaptation of \cite[Theorem 4.1]{norris1999smoluchowski}.

\medskip
Denote by $d$ some metric on $\mathcal{M}$, the set of finite measures on $\R_+$, which is compatible with the topology of weak convergence, i.e., 
\begin{equation} \label{eq:definition_metric_weak_convergence}
d(\mu_n, \mu)\to 0 \quad \mbox{if and only if} 
\quad \la f, \mu_n\ra \to \la f, \mu \ra 
\end{equation}
for all bounded continuous functions $f: \R_+\to \R$. We choose $d$ so that $d(\mu, \mu') \leq \|\mu-\mu'\|$ for al $\mu, \mu'\in \mathcal{M}$.

\begin{theorem}[Second mean-field limit] \label{th:mean-field-limit-complete}
Let $K$ be a model kernel and let $\mu_0$ be a measure on $\R_+$. Assume that for $\varphi(\o)=\o+1$ it holds
$$\K \leq \phiall$$
and that $\la \varphi, \mu_0\ra<\infty$ and $\la \varphi^2, \mu_0\ra<\infty$. Denote by $(\mu_t)_{t<T}$ the maximal strong solution to \eqref{eq:isotropic_4_wave_equation_weak} provided by Theorem \ref{eq:existence_solutions_kinetic_4wave}. Let $(X^n_t)_{n\in \NN}$ be a sequence of instantaneous coagulation-fragmentation particle system, with jump kernel $K$. Suppose that
$$d(\varphi X^n_0, \varphi \mu_0) \to 0$$
as $n\to \infty$. Then, for all $t<T$, 
$$\sup_{s\leq t} d(\varphi X^n_s, \varphi \mu_s) \to 0$$
in probability, as $n\to \infty$.
\end{theorem}

The proof of this theorem will be an adaptation of \cite[Theorem 4.4]{norris1999smoluchowski}.
\medskip

Many mathematical works have been devoted to the study of the coagulation-fragmentation equation. We base our work on \cite{norris1999smoluchowski} and \cite{norris2000cluster} but the reader is also referred to \cite{eibeck2000approximative}, \cite{escobedo2005self}, \cite{laurenccot2002discrete}, \cite{wagner2005explosion}, as an example.

\paragraph{C. Applications}
For the physical applications we consider $K$ given by expression \eqref{eq:bestK}, i.e. $\K=(\o_1\o_2\o_3)^{\lambda/3}$, which is submultiplicative (since $\o^\lambda \leq \o+1, \;\lambda\in[0,1]$). 

If $\lambda\in[0,3)$ then we can apply all the previous theorems. For the case $\lambda=3$ the theorems also apply with the exception of the first mean-field limit, Theorem \ref{th:mean_field_limit_unbounded_kernel}.

Here are some examples:
\begin{itemize}
\item \textit{Langmuir waves in isotropic plasmas and spin waves}: $\beta=2$, $\alpha=2$, so $\lambda=1$ (the dimension is $N=3$).
\item \textit{Shallow water} (isotropic in a flat bottom, \cite{zakharov1999statistical}): $\beta=2$, $\alpha=1$, so $\lambda=3$  (dimension $N=2$).
\item \textit{Waves on elastic plates}: $\beta=3$, $\alpha=2$, so $\lambda=2$ (dimension $N=2$).
\end{itemize}

However, these results cannot be applied to other systems like  gravity waves on deep water, nonlinear optics and Bose-Einstein condensates.

\subsection{Some notes on  the physical theory of Wave Turbulence}
\label{sec:brief_account}

The theory of Wave Turbulence is a relatively recent field where most of the results are due to physicists. Next, we present some concepts of the theory extracted from \cite{zakharov2004one,zakharov1992kolmogorov,nazarenko2011wave}, \cite[Entry turbulence]{scott2006encyclopedia}. All the results are formal and require a rigorous mathematical counterpart.

Wave turbulence is formed by the so-called weak wave turbulence (whose central object is the kinetic wave equation) and the so-called `coherent structures'.

Wave turbulence takes place on the onset of weakly non-linear dispersive waves. The assumption on weak non-linearity allows the derivation of the kinetic wave equation  of which \eqref{eq:kinetic_wave_equation} is an example for the case of 4 interacting waves. In the general case, $N$ waves interact in resonant sets transferring energy.

Differences between physical systems are given by the dimension of the system, the number of interacting waves and the medium itself (which is described by the dispersion relation and the interaction coefficient).

\paragraph{A. Derivation of the wave kinetic equation and the Cauchy theory.} There is not a rigorous mathematical derivation and Cauchy theory for the kinetic wave equation. In this work we prove existence and uniqueness of solutions for the isotropic weak 4-wave kinetic equation in some restricted setting.
\begin{itemize} 
	\item \emph{General procedure:} in \cite[Section 6.1.1]{nazarenko2011wave} it is given a scheme of the general procedure to derive the kinetic wave equation. We do not reproduce here the explanation there but point at some of the key steps:
	\begin{itemize}
		\item the starting point is a nonlinear wave equation (mostly written in Hamiltonian form);
		\item then the equation is written in Fourier space in $\kvec$ using the interaction representation between waves;
		\item using the weakness of the nonlinearity hypothesis, a perturbation analysis is done expanding around a small nonlinearity parameter;
		\item perform statistical averaging.
 \end{itemize}	
	\item \emph{Example: shallow water}. In the case of shallow (or deep water) the vertical coordinate is considered to be
	$$-h<z<\eta(\mathbf{r}), \qquad\mathbf{r}=(x,y)$$
	and the velocity field $V$ is incompressible and a potential field,
	$$\mbox{div } V=0, \qquad V= \nabla \Phi$$
	where the potential satisfies the Laplace equation
	$$\Delta \Phi=0$$
	with boundary conditions
	$$\Phi|_{z=\eta}=\Psi(\mathbf{r},t), \qquad\Phi|_{z=-h}=0.$$
	The Hamiltonian is consider to by the sum $H=T+U$ of kinetic and potential energies defined as follows:
	\begin{eqnarray*}	
		T&=&\frac{1}{2}\int dr \int^\eta_{-h}(\nabla \Phi)^2dz,\\
		U&=& \frac{1}{2}g \int \eta^2 dr + \sigma \int \lp \sqrt{1+(\nabla \eta)^2}-1 \rp dr
	\end{eqnarray*}
	where $g$ is the acceleration of gravity and $\sigma$ is the surface tension coefficent. Zakharov \cite{zakharov1998weakly} derived the equations of motion for $\eta$ and $\Psi$ as
	$$\frac{\partial \eta}{\partial t}=\frac{\delta H}{\delta \Psi}, \qquad\frac{\partial \Psi}{\partial t}= - \frac{\delta H}{\delta \eta}.$$
	In \cite{zakharov1999statistical}, Zakharov derives the kinetic wave equation for shallow and deep water starting from these equations.
	\item \emph{The delta distribution}. One of the main issues to study the validity of the kinetic wave equation is the presence of the two delta distributions that make sure that the energy and the total momentum are conserved. 
	\item \emph{$N$-waves}. At the beginning of this work the 4-wave equation was presented. In the general case, the kinetic equation will correspond to $N$ interacting waves, where $N$ is the minimal number such that the interaction operator is non-zero, i.e., such that
		\begin{enumerate}
			\item the $N$-wave resonant conditions are satisfied for a non-trivial set of wave vectors (here `non-trivial set' must be made precise):
			\begin{eqnarray*}
			\o(\kvec_1)\pm\o(\kvec_2)\pm\hdots\pm\o(\kvec_N)&=&0;\\
			\kvec_1 \pm \kvec_2 \pm \hdots \pm \kvec_N;
			\end{eqnarray*}
			\item the $N$-wave interaction coefficient $\overline{T}$ must be non-zero over this set.
		\end{enumerate}
	
\end{itemize}

\paragraph{B. The Kolmogorov-Zakharov (KZ) spectra.} The Kolmogorov-Zakharov spectrum corresponds to steady states of the system.
\begin{itemize}
	\item \emph{Derivation, validity (locality) and stability.} The derivation of the KZ spectrum is explained in \cite[Chapter 3]{zakharov1992kolmogorov}. For the derivation, only the homogeneity index of the interaction coefficient $\overline{T}$ is needed. However, the validity of the KZ spectrum depends on the condition of `locality', i.e., that only waves with similar wavelength interact. This condition is translated in the finiteness of the interaction integral (see \cite{zakharov1992kolmogorov}  for more details) and it does depend on the particular shape of $\overline{T}$. On the other hand, one should check the stability of the KZ to small perturbations.
	\item \emph{Case of shallow water:}
	for this case, the corresponding Kolmogorov-Zakharov solutions are (\cite{zakharov1999statistical}):
$$n^{(1)}_k \sim k^{-10/3}$$
$$n^{(2)}_k \sim k^{-3}.$$
Observe that there are two solutions; the first one corresponds to the energy flux and the second to the flux of action (corresponding to the waveaction). 
\end{itemize}

\paragraph{Historical note.} The kinetic wave equation was first derived by Nordheim in 1928 \cite{nordheim1928kinetic} in the context of a Bose gas and by Peierls \cite{peierls1929kinetischen} in 1929 in the context of thermal conduction in crystals.

\paragraph{C. Some examples.}

We have already seen the case of shallow water, but there are many more examples.

The Majda-McLaughlin-Tabak model is explained in \cite{zakharov2004one} in dimension 1 where the dispersion relation is given by
$$\omega(\kvec) =k^{\alpha}, \quad \alpha>0$$
where $k=\|\kvec\|$ and
\begin{equation} \label{eq:T_in_MMTmodel} 
	T_{123k}=  (k_1 k_2 k_3 k)^{\beta/4}
\end{equation}
for some $\beta \in \R$. The particular case  $\alpha=\frac{1}{2}$ corresponds to the Majda-McLaughlin-Tabak (MMT) model. 

We  have a four-wave interaction process with \textbf{resonant conditions}:
$$\left\{\begin{array}{l}
\kvec_1+\kvec_2 =\kvec_3+\kvec\\
|k_1|^{1/2}+|k_2|^{1/2}= |k_3|^{1/2}+ |k|^{1/2} .
\end{array}
\right.$$

In this case wave numbers that are non-trivial solutions to these conditions cannot have all the same sign. Moreover, non-trivial solutions can be parametrized by a two parameter family $A$ and $\xi>0$:
\begin{equation} \label{eq:parametrization_action_set}
k_1=-A^2\xi^2, \quad k_2=A^2(1+\xi +\xi^2)^2, \quad k_3= A^2(1+\xi)^2, \quad k=A^2\xi^2(1+\xi)^2.
\end{equation}

When $\beta=0$ the collision rate is bounded. In \cite{zakharov2004one} the authors obtain the following Kolmogorov-type solutions for $\alpha=1/2$ and $\beta=0$:
\begin{eqnarray}
	n & \sim & |k|^{-5/6}\\
	n & \sim & |k|^{-1} \, .
\end{eqnarray}

\medskip
The derivation of the kinetic wave equation is done from the equation
$$i \frac{\partial \psi}{\partial t} = \underbrace{\left| \frac{\partial}{\partial x}\right|^\alpha \psi}_{\mbox{dispersive}} + \lambda \underbrace{\left| \frac{\partial}{\partial x}\right|^{\beta/4} \lp  \left| \left| \frac{\partial}{\partial x}\right|^{\beta/4}\psi \right|^2  \left| \frac{\partial}{\partial x}\right|^{\beta/4}\psi \rp}_{\mbox{non-linearity}} \quad \lambda = \pm1$$
where $\psi(x,t)$ denotes a complex wave field.

\bigskip

Other examples in wave turbulence are (taken from \cite{nazarenko2011wave}):
\begin{itemize}
	\item 4-wave examples
		\begin{itemize}
			\item surface gravity waves; $N=2$, $\alpha=1/2$, $\beta=3$;
			\item langmuir waves in isotropic plasmas, spin waves; $N=3$, $\alpha=2$, $\beta=2$;
			\item waves on elastic plates: $N=2$, $\alpha=2$, $\beta=3$;
			\item Bose-Einstein condensates and non-linear optics: $\alpha=2$, $\beta=0$;
			\item Gravity waves on deep water: $N=2$, $\alpha=1/2$, $\beta=3$.
		\end{itemize}
	\item 3-wave examples
		\begin{itemize}
			\item capillary waves: $N=2$, $\alpha=3/2$;
			\item acoustic turbulence, waves in isotropic elastic media; $N=3$, $\alpha=1$;
			\item interval waves in stratified fluids: $N=1$, $\alpha=-1$;
		\end{itemize}
	\item other examples
		\begin{itemize}
			\item Kelvin waves on vortex filaments: $N=1$, 6-wave interaction, $\alpha=2$.
		\end{itemize}
\end{itemize}

	\section{Existence of solutions for unbounded kernel}


In this section we will follow the steps in \cite[Theorem 2.1]{norris1999smoluchowski} (see Remark \ref{rem:strategy}).

\begin{remark} We make some comments about Theorem \ref{eq:existence_solutions_kinetic_4wave}:
\begin{enumerate}
	\item The statement is correct even if 
$$\K \leq C\phiall$$
for some positive constant $C<\infty$. This only changes the $\zeta(\mu_0)$ into
$$\zeta(\mu_0)=\la \varphi^2, \mu_0 \ra^{-1} C^{-1} \la \varphi, \mu_0 \ra^{-1}.$$
Also notice that by scaling time, we can eliminate the multiplicative constant. 
	\item Notice that in the coagulation case, existence of strong solutions is assured for times $T'= \la \varphi^2, \mu_0\ra^{-1}$. We expect that in the 4-wave equations we can assure existence of strong solutions for larger times. The reason that we do not get that is because when bounding \eqref{eq:estimate_time_strong_sol}, we ignore some negative factors.
	\item We will need to use that for $\varphi(\o)=\o +1$, it holds that for any local solution $(\mu_t)_{t<T}$
\begin{equation} \label{eq:bounds_for_varphi}
	\la \varphi, \mu_t\ra \leq \la \varphi, \mu_0\ra \quad
	\mbox{for all }t<T.
\end{equation}
This is a condition for $\mu_t$ being a solution (see Definition \ref{def:solutions}). Notice that in particular strong solutions fulfilled this condition automatically as they are conservative (this is explained in expression \eqref{eq:strong_solutions_is_conservative}).
	\item We could have defined our set of test functions as including also measurable functions with linear growth (and in an unbounded interval). This way the theorem works the same and we would have that $\la \varphi, \mu_t \ra = \la \varphi, \mu_0\ra$ for all $t$ where the solution exists, i.e., for that particular set of test functions we would only consider conservative solutions.
	\item A main difference with the result obtained in \cite{norris1999smoluchowski} and \cite{norris2000cluster} is that we do not allow $K$ to blow up at zero. 
\end{enumerate}
\end{remark}

\subsection{Proof of Theorem \ref{eq:existence_solutions_kinetic_4wave}}
The rest of this section will consist on the proof of this theorem, which we will split in different propositions. We will follow the idea and structure as in \cite[Theorem 2.1]{norris1999smoluchowski}. We want to apply an iterative scheme on the equation to prove existence of solutions and for that we need estimates on $\| Q(\mu) \|$ and $\|Q(\mu)-Q(\mu')\|$, which, unfortunately, are unavailable in our present case for unbounded kernels. To sort this problem, we will consider an auxiliary process that approximates our looked for solution and that operates on bounded sets. 

This auxiliary process will take the form $(X^B_t, \Lambda^B_t)_{t\geq 0}$ for some bounded set $B$. $\Lambda^B_t$ gives an upper estimate of the effect on $X^B_t$ of the particles outside $B$ and $X^B_t$ will be a lower bound for our process in $B$.

Let $B\subset [0,\infty)$ be bounded. Denote by $\mathcal{M}_B$ the space of finite signed measures supported on $B$. We define $L^B: \mathcal{M}_B \times \mathbb{R} \rightarrow \mathcal{M}_B\times \R$ by the requirement:
\begin{eqnarray*}
\langle (f, a), L^B(\mu, \lambda) \rangle &=& \frac{1}{2} \int_D  (f(\o_1+\o_2-\o_3) \mathbb{1}_{\o_1+\o_2-\o_3\in B} + a\varphi(\om) \mathbb{1}_{\om \notin B} \nonumber\\
&& \qquad + f(\o_3) -f(\o_1) -f(\o_2) ) \K \mu(d\o_1) \mu(d\o_2)\mu(d\o_3) \nonumber\\
&+&(\lambda^2+ 2\lambda \la \varphi ,\mu \ra) \int^\infty_0 \lp a\varphi(\o)- f(\o) \rp \varphi(\o) \mu(d\o)\
\end{eqnarray*}
for all bounded measurable functions $f$ on $(0,\infty)$  and all $a\in \R$ where $D= \{ \R^3_+ \cap \om\geq 0\}$.
We used the notation $\la (f,a), (\mu, \lambda) \ra = \la f, \mu\ra + a \lambda$.

Consider the equation
\begin{equation} \label{eq:auxiliary_equation_existence_proof}
(\mu_t, \lambda_t) = (\mu_0, \lambda_0) + \int^t_0 L^B(\mu_s, \lambda_s)\, ds.
\end{equation}

We admit as a \textit{local solution} any continuous map
$$t \mapsto (\mu_t, \lambda_t):[0,T] \to \mathcal{M}_B \times \R$$
where $T\in (0,\infty)$, which satisfies equation \eqref{eq:auxiliary_equation_existence_proof} for all $t\in[0,T]$.

\begin{proposition}[Existence for the auxiliary process]
\label{prop:existence_auxiliary_process}
Suppose $\mu_0\in \mathcal{M}_B$ with $\mu_0 \geq 0$ and that $\lambda_0\in[0,\infty)$. The equation \eqref{eq:auxiliary_equation_existence_proof} has a unique solution $(\mu_t,\lambda_t)_{t\geq 0}$ starting from $(\mu_0,\lambda_0)$. Moreover, $\mu_t\geq 0$ and $\lambda_t\geq 0$ for all $t$.
\end{proposition}

The proof is obtained by adapting the one in \cite[Proposition 2.2]{norris1999smoluchowski}.
\begin{proof}
By assumption \eqref{eq:bounds_on_K} it holds that for $\varphi(\o)=\o+1$
$$\K \leq \varo\vart\varth.$$
Observe that $\varphi \geq 1$. By a scaling argument we may assume, without loss, that 
$$\la \varphi, \mu_0 \rangle + \lambda_0 \leq 1,$$
which implies that 
$$\|\mu_0\| + |\lambda_0| \leq 1.$$

We will show next  by a standard iterative scheme, that there is a constant $T>0$ depending only on $\varphi$ and $B$, and a unique local solution $(\mu_t, \lambda_t)_{t\leq T}$ starting from $(\mu_0, \lambda_0)$. Then we will see that $\mu_t\geq 0$ for all $t\in [0,T]$.

This will be enough to prove the proposition: if we put $f=0$ and $a=1$ in \eqref{eq:auxiliary_equation_existence_proof} we get
\begin{eqnarray*}
\frac{d}{dt}\lambda_t &=& \frac{1}{2} \int_D \varphi(\om)\mathbb{1}_{\om\notin B}\K \muo \mut \muth \\
&&\quad+ (\lambda^2 + 2\lambda\la \varphi, \mu \ra) \int^\infty_0 \varom^2\muom.
\end{eqnarray*}

So, since $\mu_t \geq 0$, we deduce that $\lambda_t\geq 0$ for all $t\in [0,T]$. Next, we put $f=\varphi$ and $a=1$ to see that
\begin{eqnarray} \label{eq:expression_equal_0}
\frac{d}{dt}\la \varphi, \mu_t \ra +\lambda_t &=& \frac{1}{2}\int_D (\varphi(\om)+\varth-\varo-\vart)\\ \nonumber
&&\qquad\quad \times\K\muo\mut\muth =0
\end{eqnarray}
which is zero given that $\varphi(\o)=\o+1$. Therefore,
$$\|\mu_T\|+ |\lambda_T| \leq \la \varphi, \mu_T \ra + \lambda_T = \la \varphi, \mu_0 \ra + \lambda_0 \leq 1.$$ 
We can now start again from $(\mu_T, \lambda_T)$ at time $T$ to extend the solution to $[0,2T]$, and so on, to prove the proposition.

\medskip
We use the following norm on $\mathcal{M}_B \times \R$:
$$\|(\mu, \lambda)\| = \|\mu\| + |\lambda|.$$

Note the following estimates: there is a constant $C= C(\varphi, B)<\infty$ such that for all $\mu, \mu'\in \mathcal{M}_B$ and all $\lambda, \lambda'\in \R$
\begin{eqnarray}\label{eq:estimate1_auxiliary_process}
\|L^B(\mu, \lambda)\| &\leq & C\|(\mu, \lambda)\|^3\\
\|L^B(\mu, \lambda) - L^B(\mu',\lambda') \| &\leq& C \bigg( \|\mu-\mu'\| \lp \|\mu\|^2 +\|\mu\|\|\mu'\|+\|\mu'\|^2\rp \label{eq:estimate2_auxiliary_process}\\
&&+(|\lambda|+|\lambda'|) |\lambda-\lambda'| \|\mu \|+ |\lambda'|^2\|\mu-\mu'\| \nonumber\\ 
&&\quad + |\lambda-\lambda'|\|\mu\|^2 + |\lambda'|\lp\|\mu\| \|\mu-\mu'\| + \|\mu'\| \|\mu'-\mu\|\rp \bigg) \nonumber
\end{eqnarray}

Observe that we get these estimates because we are working on a bounded set $B$.

We turn to the iterative scheme. Set $(\mu^0_t, \lambda^0_t) = (\mu_0, \lambda_0)$ for all $t$ and define inductively a sequence of continuous maps
$$t\mapsto (\mu^n_t, \lambda^n_t):\, [0,\infty)\to \mathcal{M}_B\times \R$$
by
$$(\mu_t^{n+1},\lambda^{n+1}_t)=(\mu_0, \lambda_0) + \int^t_0 L^B(\mu_s^{n},\lambda_s^{n}) \, ds.$$
Set
$$f_n(t) = \|(\mu^n_t, \lambda^n_t)\|$$
then $f_0(t)=f_n(0) = \|(\mu_0, \lambda_0) \|\leq 1$ and by the estimate \eqref{eq:estimate1_auxiliary_process} we have that
$$f_{n+1}(t) \leq 1+ C\int^t_0 f_n(s)^3\, ds.$$
Hence
$$f_n(t) \leq (1-2Ct)^{-1/2} \quad\mbox{for } t < (2C)^{-1}.$$
This last assertion is checked by induction. Suppose that it holds for $n$ then
$$f_{n+1}(t) \leq 1+ C\int^t_0 (1-2Cs)^{-3/2}\, ds = 1+ (1-2Cs)^{-1/2} |_{s=0}^{s=t}.$$
Therefore, for all $n$ setting $T= (4C)^{-1}$, we have
\begin{equation} \label{eq:estimate3_auxiliary_function}
\|(\mu^n_t, \lambda^n_t) \|\leq \sqrt{2}\qquad t\leq T.
\end{equation}
Next set $g_0(t) = f_0(t)$ and for $n\geq 1$
$$g_n(t) = \|(\mu^n_t, \lambda^n_t) - (\mu^{n-1}_t, \lambda^{n-1}_t)\|.$$
By estimates \eqref{eq:estimate2_auxiliary_process} and \eqref{eq:estimate3_auxiliary_function}, there is a constant $C=C(B, \varphi) <\infty$ such that
$$g_{n+1}(t) \leq C\int^t_0 g_n(s) \, ds \quad t\leq T.$$
Hence by the usual arguments (Gronwall, Cauchy sequence), $(\mu^n_t, \lambda^n_t)$ converges in $\mathcal{M}_B\times \R$ uniformly in $t\leq T$, to the desired local solution, which is also unique. Moreover, for some constant $C<\infty$ depending only on $\varphi$ and $B$ we have
$$\|(\mu_t, \lambda_t)\|\leq C \qquad t \leq T.$$

\medskip
Finally, we are left to check that $\mu_t\geq 0$. For this, we need the following result:

\begin{proposition} \label{prop:auxiliary_prop}
Let
$$(t,\o) \mapsto f_t(\o):[0, T]\times B \to \R$$
be a bounded measurable function, having a bounded partial derivative $\partial f/ \partial t$. Then, for all $t\leq T$,
$$
\frac{d}{dt}\la f_t, \mu_t \ra = \la \partial f/\partial t, \mu_t \ra + \la (f_t, 0), L^B(\mu_t, \lambda_t) \ra.$$
\end{proposition}
The proof is a straightforward adaptation of the same Proposition (with different $L^B$) in \cite[Proposition 2.3]{norris1999smoluchowski}.

For $t\leq T$, set 
$$\theta_t(\o_1) = \exp\int^t_0 \lp \int_{\R^2_+ \cap (\o_1 + \o_2\geq \o_3)} \K \mu_s(d\o_2) \mu_s(d\o_3) + \lp \lambda^2_s + 2\lambda_s \la \varphi, \mu_s\ra \rp\varo \rp ds$$
and define $G_t : \mathcal{M}_B \to \mathcal{M}_B$ by
\begin{eqnarray*}
\la f, G_t(\mu) \ra &=& \frac{1}{2} \int_D \lp (f\theta_t) (\om) \mathbb{1}_{\om\in B} + (f\theta_t)(\o_3) \rp\\
&& \qquad \times \K \theta_t(\o_1)^{-1} \theta_t(\o_2)^{-1}\theta_t(\o_3)^{-1}\\
&&\qquad \times \muo \mut\muth
\end{eqnarray*}
Note that $G_t(\mu) \geq 0$ whenever $\mu\geq 0$ and for some $C=C(\varphi, B)<\infty$ we have
\begin{eqnarray}
\|G_t(\mu)\|&\leq& C\|\mu\|^3\\
\|G_t(\mu)-G_t(\mu')\|&\leq& C \|\mu-\mu'\|\lp \|\mu\|^2 + \|\mu'\|\|\mu\|+ \|\mu'\|^2\rp.
\end{eqnarray}
Set $\tilde \mu_t = \theta_t \mu_t$. By Proposition \ref{prop:auxiliary_prop}, for all bounded measurable function $f$ we have
$$\frac{d}{dt}\la f, \tilde \mu_t \ra = \la f \frac{\partial \theta}{\partial t}, \mu_t \ra + \la (f\theta_t, 0), L^B(\mu_t, \lambda_t) \ra$$
so, using the symmetry of $\o_1$ and $\o_2$ in $L^B$ we get
\begin{eqnarray}
\frac{d}{dt}\la f, \tilde \mu_t \ra &=& \la f, G_t(\tilde \mu_t)\ra.
\end{eqnarray}
Thus, the function $\theta_t$ is simply designed as an integrating factor, which removes the negative terms appearing in $L^B$. 

Define inductively a new sequence of measures $\tilde \mu^n_t$ by setting $\tilde \mu^0_t =\mu_0$ and for $n\geq 0$
$$\tilde \mu^{n+1} = \mu_0 + \int^t_0 G_s(\tilde \mu^n_s) \, ds.$$
By an argument similar to that used for the original iterative scheme, the proof is completed: we can show, first, and possibly for a smaller value of $T>0$, but with the same dependence, that $\|\tilde \mu^n_t\|$ is bounded, uniformly in $n$, for $t\leq T$, and then that $\|\tilde \mu^n_t - \tilde \mu_t\|\to 0$ as $n\to \infty$. Since $\tilde \mu^n_t \geq 0$ for all $n$, we deduce $\tilde \mu_t\geq 0$ and hence $\mu_t \geq 0$ for all $t\leq T$. 
\end{proof}

We fix now $\mu_0 \in \mathcal{M}$ with $\mu_0\geq 0$ and $\la \varphi, \mu_0\ra <\infty$. For each bounded set $B\subset [0,\infty)$, let
\begin{equation}\label{eq:initial_data_auxiliary_process}
\mu^B_0 =\mathbb{1}_{B}\mu_0, \qquad\lambda^B_0=\int_{[0,\infty)\backslash B}\varphi(\o)\mu_0(d\o)
\end{equation}
and denote by $(\mu^B_t, \lambda^B_t)_{t\geq 0}$ the unique solution to \eqref{eq:auxiliary_equation_existence_proof}, starting from $(\mu^B_0, \lambda^B_0)$, provided by Proposition \ref{prop:existence_auxiliary_process}. We have that for $B\subset B'$, 
$$\mu^B_t \leq \mu^{B'}_t, \qquad \la \varphi, \mu^B_t \ra + \lambda^B_t = \la \varphi, \mu^{B'}_t \ra + \lambda^{B'}_t.$$
The inequality will be proven in Proposition \ref{prop:bounds_on_auxiliary_process1} and the equality  is consequence of expression \eqref{eq:expression_equal_0} and the fact that
$$\langle \varphi, \mu^B_0\ra + \lambda^B_0= \la \varphi, \mu^{B'}_0 \ra + \lambda^{B'}_0$$
by expression \eqref{eq:initial_data_auxiliary_process}.

Moreover, it holds that for any local solution $(\nu_t)_{t<T}$ of the 4-wave kinetic equation \eqref{eq:isotropic_4_wave_equation_weak}, for all $t<T$, 
\begin{equation} \label{eq:bounds_compare_sol_aux_sol}
\mu^B_t \leq \nu_t, \qquad \la \varphi, \mu^B_t\ra + \lambda^B_t \geq \la \varphi, \nu_t\ra.
\end{equation}
We prove the first inequality in Proposition  \ref{prop:bounds_on_auxiliary_process2}. The second inequality is consequence of
\begin{equation} \label{eq:bounds_auxiliary}
\la \varphi, \nu_t \ra \leq \la \varphi,\mu_0 \ra \leq \la \varphi, \mu_0 \ra + \lambda^B_0 = \la \varphi, \mub \ra + \lambb. 
\end{equation}

We now show how these facts lead to the proof of Theorem 
\ref{eq:existence_solutions_kinetic_4wave}.
\medskip
Set $\mu_t = \lim_{B \uparrow [0,\infty)}\mu^B_t$ and $\lambda_t = \lim_{B\uparrow[0,\infty)} \lambda^B_t$. Note that
$$\la \varphi, \mu_t\ra= \lim_{B\uparrow[0,\infty)}\la \varphi, \mu^B_t \ra\leq \la \varphi, \mu_0\ra <\infty.$$
So, by dominated convergence, using that $K$ is submultiplicative, for all bounded measurable functions $f$,
$$\int_D f(\om) \mathbb{1}_{\om\notin B} \K \mu^B_t(d\o_1)\mu^B_t(d\o_2)\mu^B_t(d\o_3)\to 0,$$
and we can pass to the limit in \eqref{eq:auxiliary_equation_existence_proof} to obtain
\begin{eqnarray*}
\frac{d}{dt}\la f, \mu_t \ra &=& \frac{1}{2}\int_D (f(\om)+f(\o_3)-f(\o_1)-f(\o_2))\\
&& \qquad\quad \times\K\mu_t(d\o_1)\mu_t(d\o_2)\mu_t(d\o_3)\\
&& -(\lambda^2_t+2\lambda_t\la \varphi,\mu_t\ra )\la f\varphi,\mu_t\ra .
\end{eqnarray*}
For any local solution $(\nu_t)_{t<T}$, for all $t<T$,
$$\mu_t\leq \nu_t, \qquad \la \varphi, \mu_t \ra + \lambda_t \geq \la \varphi, \nu_t \ra.$$
Hence, if $\lambda_t =0$ for all $t<T$, then $(\mu_t)_{t<T}$ is a local solution and, moreover, is the only local solution on $[0, T)$. If $(\nu_t)_{t<T}$ is a strong local solution, then
$$\int^t_0\la \varphi^2, \mu_s\ra\, ds \leq \int^t_0\la \varphi^2, \nu_s\ra\, ds <\infty$$
for all $t<T$; this allows us to pass to the limit in \eqref{eq:auxiliary_equation_existence_proof} to obtain
\begin{equation} \label{eq:lambda_to_be_zero}
\frac{d}{dt}\lambda_t = (\lambda_t^2+ 2\lambda_t\la \varphi, \mu_t\ra)  \la\varphi^2, \mu_t\rangle
\end{equation}
and to deduce from this equation that $\lambda_t=0$ for all $t<T$. It follows that $(\nu_t)_{t<T}$ is the only local solution on $[0, T)$. For any local solution $(\nu_t)_{t<T}$,
\begin{eqnarray}\label{eq:strong_solutions_is_conservative}
\int^\infty_0 \o \mathbb{1}_{\o\leq n} \nu_t (d\o) &=& \int^\infty_0 \o \mathbb{1}_{\o \leq n} \nu_0(d\o)\\\nonumber
&+& \frac{1}{2} \int^t_0 \int_D \left\{(\om) \mathbb{1}_{\om  \leq n}+\o_3\mathbb{1}_{\o_3 \leq n}-\o_1\mathbb{1}_{\o_1 \leq n}-\o_2\mathbb{1}_{\o_2 \leq n} \right\}\\\nonumber
&&\qquad \times \K \nu_s\doo \nu_s\dotw \nu_s\doth .
\end{eqnarray}
Hence, if $(\nu_t)_{t<T}$ is strong we have that
$$\int^t_0 \la \omega^2, \nu_s\ra \, ds \leq \int^t_0 \la \varphi^2, \nu_s\ra \, ds  <\infty.$$
Then, by dominated convergence, the second term on the right tends to $0$ as $n\to \infty$, showing that $(\nu_t)_{t<T}$ is conservative.

Suppose now that $\la \varphi^2, \mu_0\ra < \infty$ and set $T= \la \varphi^2, \mu_0 \ra^{-1}\la \varphi, \mu_0\ra^{-1}$. For any bounded set $B\subset [0,\infty)$, we have
\begin{eqnarray} \nonumber
\frac{d}{dt} \la \varphi^2, \mu^B_t \ra &\leq& \frac{1}{2} \int_D \left\{ \varphi(\om)^2+\varth^2-\varo^2-\vart^2 \right\}\\\nonumber
&&\qquad \times\K \mub\doo\mub\dotw\mub\doth\\ \label{eq:estimate_time_strong_sol}
&\leq & \int_D \varo^2\varphi(\o_2)^2\varth \mub\doo\mub\dotw\mub\doth\\ \nonumber
&\leq &\la \varphi, \mub \ra \la\varphi^2, \mu^B_t \ra^2\\\nonumber
& \leq &\la \varphi, \mu_0 \ra \la\varphi^2, \mu^B_t \ra^2
\end{eqnarray}
 where we used that
\begin{eqnarray*}
(\om+1)^2+ (\o_3+1)^2 - (\o_1+1)^2-(\o_2+1)^2 &=& (\tilde \o_1 + \tilde \o_2 - \tilde \o_3)^2 + \tilde{\o_3}^2 - \tilde{\o_1}^2-\tilde{\o_2}^2\\
&=& 2 \tilde{\o_1}\tilde{\o_2}+ 2\tilde{\o_3} \lp \tilde{\o_3}-\tilde{\o_1}-\tilde{\o_2}\rp\\
&\leq & 2 \tilde{\o_1}\tilde{\o_2}
\end{eqnarray*}
with $\tilde \o_i = \o_i +1$, and using that in our domain $\om\geq 0$, 
so for $t<T$
$$\la \varphi^2, \mu_t \ra \leq (S- \la \varphi, \mu_0 \ra t)^{-1}$$
where $S = \la \varphi^2, \mu_0\ra^{-1}$. Hence \eqref{eq:lambda_to_be_zero} holds and forces $\lambda_t =0$ for $t<T$ as above, so $(\mu_t)_{t<T}$ is a strong local solution.


\begin{proposition}
\label{prop:bounds_on_auxiliary_process1}
Suppose $B\subset B'$ and that $(\mu^B_t, \lambda^B_t)_{t\geq0}$, $(\mu^{B'}_t, \lambda^{B'}_t)_{t\geq 0}$ are the solutions of \eqref{eq:auxiliary_equation_existence_proof} for each one of these sets corresponding to the initial data given by \eqref{eq:initial_data_auxiliary_process}. Then for all $t\geq 0$, $\mu^B_t \leq \mu^{B'}_t$.
\end{proposition}

The proof is obtained by adapting the one in \cite[Proposition 2.4]{norris1999smoluchowski}.

\begin{proof}
Set
$$\theta_t(\o_1) = \exp\int^t_0 \lp \int_{\R^2_+ \cap (\o_1+\o_2\geq\o_3)} \K \mu^B_s(d\o_2)\mu^B_s(d\o_3) + ((\lambda^B_s)^2+2\lambda^B_s\la \varphi, \mu^B_s\ra )\varo \rp ds.$$
Denote by $\pi_t = \theta_t(\mu^{B'}_t-\mu^B_t)$. Note that $\pi_0\geq 0$.
By Proposition \ref{prop:auxiliary_prop}, for any bounded measurable function $f$,

\begin{eqnarray*}
\frac{d}{dt}\la f, \pi_t\ra &=& \la f\frac{\partial\theta_t}{\partial t}, \mu^{B'}_t-\mu^B_t\ra\\
&& + \la (f\theta_t, 0), L^{B'}(\mubp, \lambp)- L^B(\mub, \lambb)\ra\\
&=& \int_D (f\theta_t) (\o_1) \K\lp \mubp\doo\mub\dotw\mub\doth-\mub\doo\mub\dotw\mub\doth\rp\\
&+& \lp (\lambb)^2+2\lambb\la \varphi, \mub\rangle \rp \int^\infty_0 (f\theta_t)(\o_1) \varphi(\o_1)(\mubp\doo-\mub\doo)\\
&+&\frac{1}{2} \int_D (f\theta_t)(\om) \K \\
&& \quad \times \lp \mathbb{1}_{\om\in B'}\mubp\doo\mubp\dotw\mubp\doth-\mathbb{1}_{\om\in B}\mub\doo\mub\dotw\mub\doth \rp\\
&+ & \frac{1}{2}\int_D(f\theta_t)(\o_3) \K \\
&& \quad \times \lp \mubp\doo\mubp\dotw\mubp\doth-\mub\doo\mub\dotw\mub\doth \rp\\
&-& \int_D (f\theta_t)(\o_1)  \K\lp \mubp\doo\mubp\dotw\mubp\doth-\mub\doo\mub\dotw\mub\doth \rp\\
&-&\lp (\lambp)^2+ 2\lambp \la \varphi,\mubp\ra \rp \int^\infty_0 (f\theta_t)(\o_1) \varo \mubp\doo\\
&+&\lp (\lambb)^2+ 2\lambb \la \varphi,\mub\ra \rp \int^\infty_0 (f\theta_t)(\o_1) \varo \mub\doo\\
&=& I\\
&+& \int_D (f\theta_t) (\o_1) \K \lp \mubp\doo\mub\dotw\mub\doth-\mubp\doo\mubp\dotw\mubp\doth\rp\\
&+& \lp  \lp (\lambb)^2+ 2\lambb \la \varphi,\mub\ra \rp-\lp (\lambp)^2+ 2\lambp \la \varphi,\mubp\ra  \rp \rp \la f\theta_t\varphi, \mubp\ra
\end{eqnarray*}
where
\begin{eqnarray*}
I&:=&\frac{1}{2} \int_D (f\theta_t)(\om) \K \\
&& \quad \times \lp \mathbb{1}_{\om\in B'}\mubp\doo\mubp\dotw\mubp\doth-\mathbb{1}_{\om\in B}\mub\doo\mub\dotw\mub\doth \rp\\
&+ & \frac{1}{2}\int_D(f\theta_t)(\o_3) \K \\
&& \quad \times \lp \mubp\doo\mubp\dotw\mubp\doth-\mub\doo\mub\dotw\mub\doth \rp.
\end{eqnarray*}

Now, squaring the equality
$$\la \varphi, \mu^B_t \ra + \lambda^B_t = \la \varphi, \mu^{B'}_t \ra + \lambda^{B'}_t$$
 we have that
 $$\lp (\lambb)^2+ 2\lambb \la \varphi,\mub\ra \rp- (\lambp)^2-2\lambp \la \varphi,\mubp\ra  = \la \varphi, \mubp\ra^2-\la\varphi, \mub\ra^2$$
 and therefore
 \begin{eqnarray*}
 \frac{d}{dt} \la f, \pi_t\ra &=& I \\
 &+&\int_{\R^3_+\backslash D} (f\theta_t) (\o_1) \phiall\\
 && \quad\lp \mubp\doo\mubp\dotw\mubp\doth-\mubp\doo\mub\dotw\mub\doth\rp\\
&+&\int_{D} (f\theta_t) (\o_1) (\phiall-\K)\\
 && \quad\lp \mubp\doo\mubp\doth\mubp\doth-\mubp\doo\mub\dotw\mub\doth\rp.
 \end{eqnarray*}
Therefore, $\pi_t$ satisfies an equation of the form
$$\frac{d}{dt} \pi_t = H_t(\pi_t)$$
where $H_t: \mathcal{M}_{B'} \to \mathcal{M}_{B'}$ and it holds $H_t(\pi)\geq 0$ whenever $\pi\geq 0$ and where we have estimates, for $t\leq 1$,
$$\|H_t(\pi)\|\leq C\|\pi\|$$
for some constant $C<\infty$ depending only on $\varphi$ and $B'$. Therefore, we can apply the same sort of argument that we used for nonnegativity to see that $\pi_t\geq 0$ for all $t\leq 1$, and then for all $t<\infty$.

Explicitly, $H_t$ is
\begin{eqnarray*}
H_t &=& \frac{1}{2} \int_D (f\theta_t)(\om) \K \\
&& \quad \times \Big(\mathbb{1}_{\om\in B'}\theta_t^{-1}(\o_1)\pi\doo\mubp\dotw\mubp\doth \\
&& \qquad+\mathbb{1}_{\om\in B'} \mub\doo\theta_t^{-1}(\o_2)\pi\dotw\mubp\doth \\
&&\qquad+ \mathbb{1}_{\om\in B}\mub\doo\mub\dotw\theta_t^{-1}(\o_3) \pi\doth \Big)\\
&+ & \frac{1}{2}\int_D(f\theta_t)(\o_3) \K \\
&& \quad \times \Big( \theta^{-1}_t(\o_1) \pi\doo\mubp\dotw\mubp\doth + \theta^{-1}_t(\o_2)\pi\dotw\mub\doo\mubp\dotw\\
&&\qquad +\theta^{-1}_t(\o_3)\pi\doth\mub\doo\mub\dotw \Big)\\
&+&\int_{\R^3_+\backslash D} (f\theta_t) (\o_1) \phiall\\
 && \quad\times\lp \mubp\doo\theta^{-1}_t(\o_2)\pi\dotw\mubp\doth+\mubp\doo\mub\dotw\theta^{-1}_t(\o_3)\pi\doth\rp\\
&+&\int_{D} (f\theta_t) (\o_1) (\phiall-\K)\\
 && \quad\times\lp \mubp\doo\theta^{-1}_t(\o_2)\pi\dotw\mubp\doth+\mubp\doo\mub\dotw\theta^{-1}_t(\o_3)\pi\doth\rp.
\end{eqnarray*}
where we have used that 
$$\mathbb{1}_{\om\in B'}\mub\doo \mub\dotw\mubp\doth= \mathbb{1}_{\om\in B}\mub\doo\mub \dotw\mubp\doth,$$
given that if $\o_1, \o_2<b$ and $\o_3>b$ then it must hold that $\om<b$.

\end{proof}

\begin{proposition} \label{prop:bounds_on_auxiliary_process2}
Suppose that $(\nu_t)_{t<T}$ is a local solution of the 4-wave kinetic equation \eqref{eq:isotropic_4_wave_equation_weak}, starting from $\mu_0$. Then, for all bounded sets $B\subset [0,\infty)$ and all $t<T$, $\mub\leq\nu_t$.
\end{proposition}

The proof is obtained by adapting the one in \cite[Proposition 2.5]{norris1999smoluchowski}.

\begin{proof}
Set $\theta_t$ as in the previous Proposition and denote $\nu^B_t = \mathbb{1}_B \nu_t$ and $\pi_t = \theta_t(\nu^B_t-\mub).$
By a modification of Proposition \ref{prop:auxiliary_prop}, we have, for all bounded measurable functions $f$,
$$\frac{d}{dt}\la f, \pi_t \ra = \la f \partial \theta/ \partial t, \nu^B_t- \mub\ra + \la f\theta_t\mathbb{1}_B, Q(\nu_t)\ra - \la (f\theta_t, 0), L^B(\mub, \lambb)\ra.$$

Now, proceeding as before we have that 
\begin{eqnarray*}
	\frac{d}{dt}\la f, \pi_t \ra &=&  
 \int_D (f\theta_t) (\o_1) \K (\nu^B_t\doo\mub\dotw\mub\doth-\mub\doo\mub\dotw\mub\doth)\\
&+& \lp (\lambb)^2+2\lambb\la \varphi, \mub\rangle \rp \int^\infty_0 (f\theta_t)(\o_1) \varphi(\o_1)\nu^B_t\doo\\
&+&\frac{1}{2} \int_D (f\theta_t)(\om) \K \\
&& \quad \times  \mathbb{1}_{\om\in B}\lp\nu_t\doo\nu_t\dotw\nu_t\doth-\mub\doo\mub\dotw\mub\doth \rp\\
&+ & \frac{1}{2}\int_D(f\theta_t)(\o_3) \K \\
&& \quad \times \lp \nu_t\doo\nu_t\dotw\nu^B_t\doth-\mub\doo\mub\dotw\mub\doth \rp\\
&-& \int_D (f\theta_t)(\o_1)  \K\lp \nu^B_t\doo\nu_t\dotw\nu_t\doth-\mub\doo\mub\dotw\mub\doth \rp\\
&=&\chi_t \int^\infty_0 (f\theta_t)(\o_1) \varphi(\o_1)\nu^B_t\doo\\
&+&\frac{1}{2} \int_D (f\theta_t)(\om) \K \\
&& \quad \times  \mathbb{1}_{\om\in B}\lp\nu_t\doo\nu_t\dotw\nu_t\doth-\mub\doo\mub\dotw\mub\doth \rp\\
&+ & \frac{1}{2}\int_D(f\theta_t)(\o_3) \K \\
&& \quad \times \lp \nu_t\doo\nu_t\dotw\nu^B_t\doth-\mub\doo\mub\dotw\mub\doth \rp\\
&+& \int_{\R^3_+ \backslash D} (f\theta_t)(\o_1) \phiall \lp\nu^B_t\doo\nu_t\dotw\nu_t\doth- \nu^B_t\doo\mub\dotw\mub\doth\rp\\
&+& \int_{D}(f\theta_t)(\o_1) (\phiall-\K) \\
&& \times\qquad \lp\nu^B_t\doo\nu_t\dotw\nu_t\doth- \nu^B_t\doo\mub\dotw\mub\doth\rp
\end{eqnarray*}
where $\chi_t = (\lambb)^2+ 2\lambb \la \varphi, \mub\ra + \la \varphi, \mub \ra^2- \la \varphi, \nu_t \ra^2 \geq 0$.

Therefore, analogously as in the previous Proposition \ref{prop:bounds_on_auxiliary_process1}, we have that
$$\frac{d}{dt}(\pi_t) = \tilde H_t(\pi_t) $$
where $\tilde{H}_t: \mathcal{M}_{B} \to \mathcal{M}_{B}$ is linear and $\tilde{H}_t(\pi)\geq 0$ whenever $\pi\geq 0$. Moreover for $t\leq 1$
$$\|\tilde H_t(\pi) \|\leq C\|\pi\|$$
for some constant $C<\infty$ depending only on $\varphi$ and $B$. 
\end{proof}

	\section{Mean-field limit}
		\subsection{The instantaneous coagulation-fragmentation stochastic process}
			\label{sec:coagulation_fragmentation_stochastic_process}

Define 
$$D= \{(\o_1, \o_2, \o_3)\in \R_+^3 \, |\, \o_1 + \o_2 \geq\o_3\}.$$

We consider $X^n_0$ a probability measure on $\R_+$ written as a sum of unit masses
$$X^n_0 = \frac{1}{n}\sum_{i=1}^n \delta_{\o_i}$$
for $\o_1,\hdots, \o_n\in \R_+$. $X^n_0$ represents a system of $n$ waves labelled by their dispersion $\o_1,\hdots, \o_n$.

We define a Markov process $(X^n_t)_{t\geq 0}$ of probability measures on $\R_+$. For each triple $(\o_i,\o_j,\o_l)\in D$ of distinct particles, take an independent exponential random time $T_{ijl}$, $i<j$, with parameter 
\begin{equation} \label{eq:jump_rate_MC}
	\frac{1}{n^2}K(\o_i, \o_j, \o_l).
\end{equation}
Set $T_{ijk}=T_{jik}$ and set $T= \min_{ijl} T_{ijl}$. Then set
$$X^n_t = X^n_0 \qquad \mbox{ for } t<T$$
and 
$$X^n_T = X^n_0 + \frac{1}{n}(\delta_{\o}+\delta_{\o_l}-\delta_{\o_i}-\delta_{\o_j})$$
with $\o= \o_i+\o_j-\o_l$.
Then begin the construction afresh from $X^n_T$. 

\medskip
We call the process $(X^n_t)_{t\geq0}$ an instantaneous $n$-coagulation-fragmentation stochastic process.

\begin{remark}
Note that we should be careful not to pick the same particle twice as one particle cannot interact with itself. Suppose that $\o_i=\o_j=\o_l$ then, the Markov Chain does not make a jump. The same happens with $\o_i=\o_l$ or $\o_j=\o_l$.
Finally the case $\o_i=\o_j$ needs to be considered. For that, we define
$$\mu^{(1)}(A\times B\times C) = \mu(A)\mu(B)\mu(C) - \mu(A\cap B) \mu(C)$$
as the counting measure of triples of particles with different particles in the first and second position. Also, define
\begin{equation}  \label{eq:counting_measure}
 \mu^{(n)}(A\times B\times C) = \mu(A)\mu(B)\mu(C) -n^{-1} \mu(A\cap B) \mu(C).
 \end{equation}
 
 Note that 
 \begin{equation} \label{eq:rescaling_property_counting_measure4}
 n^3\mu^{(n)}= (n\mu)^{(1)}.
 \end{equation}
\end{remark}

\paragraph{Generator of the Markov Chain:}
For all $F\in C_b$:
$$GF(X) =\frac{n}{2}\int_{D}\left[F(X^{\o_1, \o_2, \o_3})-F(X) \right] \K X^{(n)}(d\o_1, d\o_2, d\o_3) $$
where
$$X^{\o_1, \o_2, \o_3}= X + \frac{1}{n}\lp\delta_{\o_3} + \delta_{\o_1 + \o_2-\o_3} - \delta_{\o_1} - \delta_{\o_2} \rp.$$

\medskip
\paragraph{Interpretation of the stochastic process.} Three different particles, say $\o_1$, $\o_2$, $\o_3$ interact at a random time given by the rate \eqref{eq:jump_rate_MC}. 

The outcome of the interaction is that $\o_1$ and $\o_2$ merge and then, under the presence of $\o_3$, they split, creating a new particle $\o_3$ and another one with the rest $\o= \o_1+\o_2-\o_3$. (Coagulation-fragmentation phenomena, which takes place instantaneously).

\paragraph{The martingale formulation.}
Now, for each function $f\in C_b(\R_+)$ the Markov chain can be expressed as
\begin{equation} \label{eq:martingale_formulation}
\langle f, X^n_t\rangle = \langle f, X_0^n \rangle + M^{n,f}_t + \int^t_0 \langle f, Q^{(n)}(X^n_s)\rangle\,  ds
\end{equation}
where $(M^{n,f}_t)_{t\geq 0}$ is a martingale. Note that using \eqref{eq:rescaling_property_counting_measure4} we have that
\begin{eqnarray*}
&&\langle f, Q^{(n)}(\mu) \rangle\\
&=&\frac{1}{2}\int_{D}\frac{1}{n} (f(\o_1+\o_2-\o_3)+f(\o_3)-f(\o_1) -f(\o_2)) \frac{1}{n^2} K(\o_1,\o_2,\o_3)(n\mu)^{(1)}(d\o_1,d\o_2,d\o_3)\\
&=&\frac{1}{2}\int_{D} (f(\o_1+\o_2-\o_3)+f(\o_3)-f(\o_1) -f(\o_2))  K(\o_1,\o_2,\o_3)\mu^{(n)}(d\o_1, d\o_2,d\o_3) 
\end{eqnarray*}
from this expression it is clear why we needed to rescaled the collision frequency by $n^2$.

		\subsection{First result on mean-field limit}

 We will start working in the simpler case where $K$ is bounded and see that the unbounded case will come as a `modification' of the bounded one.

\subsubsection{Mean-field limit for bounded jump kernel}

\paragraph{Uniqueness of solutions for bounded kernel}

\begin{lemma}
\label{lem:properties_trilinear_operator}
It holds that $Q$ given in \eqref{eq:isotropic_4_wave_equation_weak}
is linear in each one of its terms and the following symmetry
$$\langle f, Q(\mu, \nu, \tau) \rangle = \langle f, Q(\nu, \mu, \tau) \rangle$$
but
\begin{eqnarray*}
\langle f, Q(\mu, \nu, \tau) \rangle &\neq& \langle f, Q(\mu, \tau, \nu) \rangle\\
\langle f, Q(\mu, \nu, \tau) \rangle &\neq& \langle f, Q(\tau, \nu, \mu) \rangle.
\end{eqnarray*}
Moreover,
\begin{equation} \label{eq:trilinearQ}
Q(\mu,\mu, \mu)-Q(\nu, \nu,\nu) = Q(\mu+\nu, \mu-\nu, \mu)+ Q(\mu+\nu, \nu, \mu-\nu) +Q(\mu, \nu, \nu-\mu)
\end{equation}

\end{lemma}
\begin{proof}
The first part of the statement is immediate from the definition. The second part will make use of this symmetry property along with the linearity in each component:
\begin{eqnarray*}
Q(\mu, \mu, \mu) - Q(\nu, \nu, \nu) &=&Q(\mu, \mu, \mu) + Q(\nu, \mu,\mu) - Q(\mu, \nu, \mu) + Q(\mu, \nu, \nu) \\
&&- Q(\mu,\nu, \nu)- Q(\nu, \nu, \nu)\\
&=& Q(\mu+\nu, \mu, \mu) + Q(\mu, \nu, \nu-\mu)- Q(\mu+\nu, \nu, \nu)\\
&=& Q(\mu+\nu, \mu, \mu) - Q(\mu+\nu, \nu,\mu) + Q(\mu+\nu, \nu,\mu)- Q(\mu+\nu, \nu, \nu)\\
&&\quad+\, Q(\mu, \nu, \nu-\mu)\\
&=& Q(\mu+\nu, \mu-\nu, \mu) + Q(\mu+\nu, \nu, \mu-\nu) + Q(\mu, \nu, \nu-\mu)
\end{eqnarray*}
\end{proof}

\begin{proposition}[Uniqueness of solutions]
\label{prop:uniqueness_solutions_kinetic_wave}
Suppose that the jump kernel in \eqref{eq:isotropic_4_wave_equation_weak} is bounded by $\Lambda$. Then for any given initial data, if there exists a solution for \eqref{eq:isotropic_4_wave_equation_weak}, then the solution is unique.
\end{proposition}
\begin{proof}
Suppose that we have $\mu_t, \nu_t\in \P(\R_+)$ solutions to \eqref{eq:isotropic_4_wave_equation_weak} with the same initial data. We will compare these solutions in the total variation norm:
$$\|\mu_t-\nu_t||_{TV} = \sup_{\|f\|_\infty =1} \langle f, \mu_t-\nu_t \rangle = \sup_{\|f\|_\infty =1} \int^t_0 \langle f, \dot \mu_t - \dot \nu_t \rangle.$$
Then by expression \eqref{eq:trilinearQ} we have that
$$\dot \mu_s - \dot \nu_s = Q(\mu_s+\nu_s, \mu_s-\nu_s, \mu_s)+ Q(\mu_s+\nu_s, \nu_s, \mu_s-\nu_s) +Q(\mu_s, \nu_s, \nu_s-\mu_s).$$
Therefore, for any $f\in C_b(\R_+)$ such that $\|f\|_\infty=1$ it holds
$$|\langle f, \dot \mu_s- \dot\nu_s \rangle | \leq 24 \Lambda \|\mu_s -\nu_s \|_{TV}.$$
Finally applying Gronwall on
$$\|\mu_t-\nu_t\|_{TV} \leq 24 \Lambda \int^t_0 \|\mu_s - \nu_s\|_{TV}\, ds$$
we have that the two solutions must coincide. 
\end{proof}

\begin{remark}
Existence of solutions for this case can be proven directly using a classical argument of iterative scheme (as done previously for the unbounded case).
\end{remark}

The following theorem is an adaptation of  part of \cite[Theorem 4.1]{norris1999smoluchowski}. Much more detail is provided here than in the original reference. To give the details, the author was much guided by an unpublished report \cite{martinccareport} that studied the homogoneous Boltzmann equation with bounded kernels.

\begin{theorem}[Mean-field limit for bounded jump kernel]
\label{th:hydrodynamic_limit}
Suppose that for a given measure $\mu_0$ it holds that
$$\la \o, X^n_0\ra \leq \la \o, \mu_0\ra$$
and that as $n\rightarrow \infty$
$$X^n_0 \rightarrow \mu_0 \qquad \mbox{weakly}$$
Assume that the kernel is uniformly bounded
$$K \leq \ \Lambda < \infty.$$

Then the sequence $(X^n)_{t\geq 0}$ converges as $n\rightarrow \infty$ in probability in $D([0,\infty) \times \P(\R_+))$. Its limit, $(\mu_t)_{t\geq 0}$ is continuous and it satisfies the isotropic  4-wave kinetic equation \eqref{eq:isotropic_4_wave_equation_weak}. In particular, for all $f\in C_b(\R_+)$
\begin{eqnarray*}
\sup_{s\leq t }\la f, X^n_t \ra & \to & \la f, \mu_t\ra,  \\
\sup_{s\leq t }|M_s^{f,n}| &\to & 0,\\
\sup_{s\leq t }\int^t_0 \la f, Q^{(n)}(X^n_s)\ra\, ds &\to & \int^t_0 \la f, Q(\mu_s) \ra\, ds 
\end{eqnarray*}
all in probability. As a consequence, equation \eqref{eq:isotropic_4_wave_equation_weak} is obtained as the limit in probability of \eqref{eq:martingale_formulation} as $n\rightarrow \infty$.

\end{theorem}

\begin{corollary}[Existence of solutions for the weak wave kinetic equation]
There exists a solution for \eqref{eq:isotropic_4_wave_equation_weak} (expressed as the limit of the $X^n_t$).
\end{corollary}

\begin{proof}
We have that the limit $(\mu_t)_{t\geq 0}$ satisfies $\la \o, \mu_t \ra \leq \la \o, \mu_0\ra$ by the following
$$\la \o \mathbb{1}_{\o\leq k}, \mu \ra = \lim_{n\to \infty}\la \o\mathbb{1}_{\o\leq k}, X^n_t \ra $$
and we have that
$$\la \o\mathbb{1}_{\o\leq k}, X^n_t \ra  \leq\la \o, X^n_t \ra  \leq \la \o, \mu_0\ra.$$
So by making $k\to \infty$ we get the bound.
\end{proof}

\subsubsection{Proof of Theorem \ref{th:hydrodynamic_limit}}
We want to take the limit in the martingale formulation \eqref{eq:martingale_formulation}. For that we will follow the following steps in \cite{norris1999smoluchowski}:
\begin{enumerate}
	\item The martingale $(M^{n, f})_{n\in\NN}$ converges uniformly in time for bounded sets to zero 
	$$\sup_{0 \leq s\leq t} |M^{n, f}_s| \rightarrow 0 \qquad \mbox{in probability}$$
	(Proposition \ref{prop:convergence_martingale}).
	\item Up to a  subsequence $(X^n_t)_{n\in\NN}$ converges weakly as $n\rightarrow \infty$ in $D([0,\infty)\times \P(\R_+)$ (Proposition \ref{prop:almost_sure_convergence_for_measures}). This will be split in three steps:
		\begin{enumerate}
			\item We will prove that the laws of the sequence $(\langle f, X^n_t \rangle)_{n\in\NN}$ are tight in $D([0,\infty), \R)$ (Lemma \ref{lem:tightness_for_the_action}).
			\item From this deduce that actually the laws of the sequence $(X^n_t)_{n\in\NN}$ itself is tight in $\mathcal{P}(D([0,\infty)\times\P(\R_+)))$ (Lemma \ref{lem:tightness_for_the_measures}).
			\item Finally use Prokhorov theorem to prove the statement.
		\end{enumerate}	
	\item Compute the limit of the trilinear term (Proposition \ref{prop:convergence_trilinear_term}). For this we will need to prove that:
		\begin{enumerate}
			\item The limit of $(X^n_t)_{t\geq 0}$ as $n\rightarrow\infty$ is uniformly in compact sets of the $t$ variable (Lemma \ref{lem:uniform_convergence_in_time}). This will be a consequence of proving that the limit itself is continuous (Lemma \ref{lem:continuity_of_limit}). 	
			\item Prove that actually in the limit we can forget about the counting measure $X^{(n)}$ and consider just the product of the three measures $X(d\o_1)X(d\o_2) X(d\o_3)$ (Lemma \ref{lem:limit_counting_measures}).
		\end{enumerate}

	\item Using the uniqueness of the wave kinetic equation, we have that all the convergent subsequences converge to the same limit. Hence the whole sequence converges; if a tight sequence has every weakly convergent subsequence converging to the same limit, then the whole sequence converges weakly to that limit (\cite{billingsley2013convergence}). 
	\item We have that the weak limit of $(X_t^n)_{n\in\NN}$ satisfies the kinetic wave equation \eqref{eq:isotropic_4_wave_equation_weak} so it is deterministic. Therefore, we actually have convergence in probability. 
		\item Finally, as an application of the functional monotone class theorem we can extend this result to functions $f\in \mathcal{B}(\R_+)$.
\end{enumerate}

\bigskip

\paragraph{Step 1: control on the martingale}
\begin{proposition}[Martingale convergence]
\label{prop:convergence_martingale}
For any $f\in C_b(\R_+)$, $t\geq 0$
$$\sup_{0 \leq s\leq t} |M^{n, f}_s| \rightarrow 0\qquad \mbox{in } L^2(\R)$$
in particular, it also converges in probability.
\end{proposition}
\begin{proof}[Proof of Proposition \ref{prop:convergence_martingale}]
We use Proposition 8.7 in \cite{darling2008differential} that ensures that
$$\E  \left[ \sup_{s\leq T} |M^{n, f}_s|^2\right] \leq 4 \E \int^T_0 \alpha^{n,f}(\mu_s) ds$$
as long as the right hand side is finite, where
\begin{eqnarray} \label{eq:alpha_previsible}
\alpha^{n,f}(\mu_s) &=& \frac{1}{2} \int_{D}\lp \frac{1}{n} (f(\o_1+\o_2-\o_3)+f(\o_3) - f(\o_1) -f(\o_2)) \rp^2 \\
&& \qquad\times\frac{1}{n^2} K(\o_1, \o_2, \o_3) (n\mu_s)^{(1)}(d\o_1,d\o_2,d\o_3) \nonumber
\end{eqnarray}
(this statement is consequence of Doob's $L^2$ inequality).
Therefore, using \eqref{eq:rescaling_property_counting_measure4} we have that
\begin{equation} \label{eq:bound_martingale}
\E  \left[ \sup_{s\leq t} |M^{n, f}_s|^2\right]  \leq \frac{1}{n}32 \|f\|^2_{\infty} \Lambda^2 t.
\end{equation}
This implies the convergence of the supremum towards 0 in $L^2$ which implies also the convergence in probability. 
\end{proof}

\bigskip
\paragraph{Step 2: convergence for the measures}
\label{sec:step2_convergence_of_measures}

\begin{proposition}[Weak convergence for the measures]
\label{prop:almost_sure_convergence_for_measures}
There exists a weakly convergent subsequence  $(X^{n_k}_t)_{k\in\NN}$ in $D([0,\infty)\times \P(\R_+))$ as $k\rightarrow\infty$.
\end{proposition}

\begin{lemma} 
\label{lem:tightness_for_the_action}
The sequence of laws  of $(\langle f, X^n_t \rangle)_{n\in\NN}$ on $D([0,\infty), \R)$ is tight.			
\end{lemma}
\begin{lemma} 
\label{lem:tightness_for_the_measures}
The laws of the sequence $(X^n_t)_{n\in\NN}$ on $D([0,\infty)\times\P(\R_+))$ is tight.
\end{lemma}

\begin{proof}[Proof of Proposition \ref{prop:almost_sure_convergence_for_measures}]
By Lemma \ref{lem:tightness_for_the_measures} we know that the laws of the sequence $(X^n_t)_{n\in\NN}$ are tight. This implies relative compactness for the sequence by  Prokhorov's theorem.
\end{proof}

\begin{proof}[Proof of Lemma \ref{lem:tightness_for_the_action}]
 We use  Theorem \ref{th:criteria_tightness_Skorokhod}. To prove the first part (i) of the Theorem we use that
$$|\la f, X^n_t \ra | = \left| \frac{1}{n}\sum^n_{i=1}f(\o^{i,n}_t)\right| \leq  \frac{1}{n} \sum_{i=1}^n|f(\o^{i, n}_t)|\leq \|f\|_\infty$$
so for all $t\geq 0$, $\la f, X^n_t \ra \in  [-\|f\|_\infty, \|f\|_\infty]$.

The second condition (ii) of the theorem will be consequence of the following inequalities:
\begin{equation} \label{eq:4bound_square_martingale}
\E \left[ \sup_{r\in[s,t)}|M_{r}^{n,f}-M_s^{n,f}|^2\right] \leq \frac{1}{n} 32 \|f\|^2_\infty \Lambda^2(t-s)
\end{equation}
 and 
\begin{equation} \label{eq:4bound_square_operator}
 \E \left[ \sup_{r\in[s,t)} \left( \int^r_s \langle f, Q^{(n)}(X^n_p) \rangle \, dp \rp^2 \right]\leq 16 \|f\|^2_\infty \Lambda^2(t-s)^2.
 \end{equation}
  which imply that
\begin{equation} \label{eq:4bound_square_measure}
\E \left[ \sup_{r\in[s,t)}|\langle f, X^n_r-X^n_s\rangle|^2\right] \leq A\lp(t-s)^2+\frac{(t-s)}{n} \rp
\end{equation}
for some $A>0$ depending only on $\|f\|_\infty$ and $\Lambda$. 

 First we use Markov's and Jensen's inequalities to get
$$\mathbb{P}(w'(\la f, X^n\ra, \delta, T) \geq \eta) \leq \frac{\mathbb{E}[w'(\la f, X^n \ra, \delta, T)]}{\eta} \leq \frac{\lp \mathbb{E}[w'(\la f, X^n\ra, \delta, T)^2] \rp^{1/2}}{\eta}.$$
($w'$ is defined in Theorem \ref{th:criteria_tightness_Skorokhod}). Now, for a given partition $\{t_i\}^n_{i=1}$,
$$\sup_{r_1, r_2\in [t_{i-1}, t_i)}|\la f, X^n_{r_1}-X^n_{r_2} \ra | \leq 2 \sup_{r\in [t_{i-1}, t_i)}|\la f, X^n_r-X^n_{t_{i-1}}\ra|.$$
Denote by $i^*$ the point where the maximum on the right hand side is attained (the number of points in each partition is always finite). Now we want to consider a partition such that $\mbox{max}_{i}|t_i-t_{i-1}|=\delta+\varepsilon$ for some $\varepsilon>0$ so
$$w'(\la f, X^n\ra, \delta, T) \leq 2 \sup_{r\in[t_{i^*-1},t_{i^*-1}+\delta+\varepsilon)}|\la f, X^n_r-X^n_{t_{i^*-1}}\ra| \quad a.s..
$$
Therefore we are just left to check that
$$\mathbb{E} \left[\sup_{r\in[s,s+\delta+\eps)}|\la f, X^n_r-X^n_s\ra|^2 \right] \leq \frac{\eta^4}{2}$$ 
which is fulfilled thanks to the bound \eqref{eq:4bound_square_measure} by taking, for example,
$$\delta = \sqrt{1+\frac{\eta^4}{2A}}-1-\varepsilon$$
for $\varepsilon$ small enough.
\end{proof}

\begin{proof}[Proof of Lemma \ref{lem:tightness_for_the_measures}]
We will use Theorem \ref{th:jakubowski_criteria} to prove this. To check condition (i), we consider the compact set $W\in \P(\R_+)$ (compact with respect to the topology induced by the weak convergence of measures) defined as
$$W:=\left\{ \tau\in \P(\R_+)\,:\, \int_{\R_+} \o\,  \tau(d\o) \leq C \right\}.$$
Consider $\lp\mathcal{L}^n\rp_{n\in\NN}$ the family of probability measures in $\mathcal{P}(D([0,\infty); W))$ which are the laws of $(X^n)_{n\in \NN}$. We have that
$$\mathcal{L}^n(D([0,\infty); W) =1 \quad \mbox{ for all } n\in \NN$$
by the conservation of the total energy and its boundedness (assumption (B1)):
$$\int_{\R_+} \o X_t^n(d\o) = \frac{1}{n}\sum_{i=0}^n \o_t^{n,i} =\frac{1}{n}\sum_{i=0}^n \o_0^{n,i} = \int_{\R_+} \o \mu_0^n(d\o) \leq C \quad \mbox{a.s..}$$

Now, to check condition (ii) we will use the family of continuous functions in $\P(\R_+)$ defined as
$$\mathbb{F}= \{ F\, :\, \mathcal{P}(\R_+) \rightarrow \R \, :\, F(\tau)= \langle f, \tau \rangle \mbox{ for some } f\in C_b(\R_+)\}.$$
This family is closed under addition since $C_b(\R_+)$ is, it is continuous in $\mathcal{P}(\R_+)$, and separates points in $\P(\R_+)$: if $F(\tau)=F(\bar \tau)$ for all $F\in \mathbb{F}$ then
$$\int_{\R_+} f(k) d(\tau-\bar \tau) (k) =0 \quad \forall f\in C_b(\R_+)$$
 hence $\tau\equiv\bar\tau$. 

So we are left with proving that for every $f\in C_b(\R_+)$ the sequence $\{\langle f, X^n \rangle\}_{n\in \NN}$ is tight. This was proven in Lemma \ref{lem:tightness_for_the_action}.
\end{proof}

\bigskip

\paragraph{Step 3: convergence for the trilinear term}
\label{sec:4convergence_trilinear_term}

\begin{proposition}[Convergence for the trilinear term]
\label{prop:convergence_trilinear_term}
It holds that
$$\int^t_0 \langle f, Q^{(n)}(X^{n_k}_s)\rangle \, ds \rightarrow \int^t_0 \langle f, Q(\mu_s, \mu_s, \mu_s)\rangle \, ds \quad \mbox{weakly.}$$
\end{proposition}

\begin{lemma}[Continuity of the limit]
\label{lem:continuity_of_limit}
The weak limit of $(X^{n_k}_t)_{t\geq 0}$ as $k\rightarrow\infty$  is continuous in time a.e..
\end{lemma}

\begin{lemma}[Uniform convergence]
\label{lem:uniform_convergence_in_time}
For all $f\in C_b(\R_+)$, it holds
$$\sup_{s\leq t}|\la f, X^{n_k}_s-\mu_s\ra|\to 0 \quad
\mbox{weakly}$$
as $k\to\infty$. 
\end{lemma}

\begin{lemma}
\label{lem:limit_counting_measures}
It holds that
$$\sup_{s\leq t} |\langle f, Q^{(n)}(X^{n_k}_s)- Q(\mu_s)\rangle| \rightarrow 0 \quad \mbox{weakly}$$
as $k\rightarrow \infty$.
\end{lemma}

\begin{proof}[Proof of Proposition \ref{prop:convergence_trilinear_term}]
By Lemma \ref{lem:limit_counting_measures} we can pass the limit inside the integral in time. 
\end{proof}

\begin{proof}[Proof of Lemma \ref{lem:continuity_of_limit}]
We have that for any $f\in C_b(\R_+)$
$$|\langle f, X^{n_k}_t \rangle- \langle f, X^{n_k}_{t-}\rangle | \leq\frac{4}{n_k}  \|f\|_{\infty}$$
 applying Theorem \ref{th:continuity_criteria_limit_Skorokhod_space} we get that $\la f, \mu_t\ra$ is continuous for any $f\in C_b(\R_+)$ and this implies the continuity of $(\mu_t)_{t\geq0}$.
\end{proof}

\begin{proof}[Proof of Lemma \ref{lem:uniform_convergence_in_time}]
We know by Lemma \ref{lem:continuity_of_limit} that the limit of $(X^{n_k})_{k\in \NN}$ is continuous. The statement is consequence of the continuity mapping theorem in the Skorokhod space (proven using the Skorokhod representation theorem \ref{th:Skorokhod_representation_theorem}) and the fact that $g(X)(t)=\sup_{s\leq t} |X|$ is a continuous function in this space.
\end{proof}

\begin{proof}[Proof of Lemma \ref{lem:limit_counting_measures}]
We abuse notation and denote by $(X^n_t)_{n\in \NN}$ the convergent subsequence.
We split the proof in two parts, we will prove for all $f\in C_b(\R_+)$:
\begin{itemize}
	\item[(i)] $\sup_{s\leq t}|\langle f, \lp Q-Q^{(n)}\rp (X^{n}_s)\rangle| \rightarrow 0$ as $n\rightarrow \infty$,
	\item[(ii)] $\sup_{s\leq t}\left|\langle f,  Q\lp X^n_s \rp -Q\lp\mu_s\rp \rangle\right| \rightarrow 0$ as $n\rightarrow \infty$. 
\end{itemize}
(i) is consequence of 
\begin{eqnarray} \nonumber
|\langle f, \lp Q-Q^{(n)}\rp (X^{n}_s)\rangle| &=& \frac{1}{2}\frac{1}{n}\int_{2\o_2\geq \o_3}\lp f(2\o_2-\o_3) + f(\o_3)-2f(\o_2)\rp \\
&&\qquad\qquad\times K(\o_2, \o_2, \o_3) X^n_s(d\o_2)X^n_s(d\o_3) \nonumber\\
&\leq&\frac{2}{n}\|f\|_{\infty}\Lambda . \label{eq:4bound_difference_trilinear_term}
\end{eqnarray}
Now, for (ii) we compute we have that
\begin{eqnarray} \nonumber
\sup_{s\leq t}\left|\langle f,  Q( X^n_s) -Q(\mu_s) \rangle\right| & \leq & \frac{1}{2} \int_{D} K(\o_1, \o_2, \o_3) \left|f(\o_1 + \o_2 -\o_3) +f(\o_3) - f(\o_1) - f(\o_2) \right| \\
&&\qquad\times \sup_{s\leq t} \left|X^n_s(d\o_1)X_s^n(d\o_2)X_s^n(d\o_3)-\mu_s(d\o_1)\mu_s(d\o_2)\mu_s(d\o_3)\right| \nonumber\\
\label{eq:4bound_trilinear_term_uniformly}
\end{eqnarray}

We conclude (ii)  with an argument analogous to Lemma \ref{lem:uniform_convergence_in_time} and the fact that
$$X^n_t \otimes X^n_t \otimes X^n_t \to \mu_t \otimes \mu_t \otimes \mu_t$$
weakly (consequence of L\'evy's continuity theorem).
\end{proof}

\subsubsection{Proof of Theorem \ref{th:mean_field_limit_unbounded_kernel} (unbounded kernel)}


\begin{remark}
The proof that we already wrote in the case of bounded kernels works here in most parts substituting $\Lambda$ by $M$ where
$$\int_{\R_+}\o X^n(d\o)\leq M= \la \o, \mu_0\ra.$$
The only places where we need to be careful are Lemmas \ref{lem:uniform_convergence_in_time} and \ref{lem:limit_counting_measures}.
\end{remark}

\begin{lemma}[Convergence of a subsequence]
There exists a subsequence $\lp X^{n_k}_t\rp_{k\in \NN}$ that converges weakly in $D([0,\infty)\times \mathcal{P}(\R_+))$ as $k\rightarrow \infty$. 
\end{lemma}
\begin{proof}
The proof is exactly the one as in Section \ref{sec:step2_convergence_of_measures} and Proposition \ref{prop:almost_sure_convergence_for_measures} using the bound on the jump kernel $K$, for example in the proof of Lemma \ref{lem:tightness_for_the_action}, in the bounds of expressions \eqref{eq:4bound_square_martingale} and \eqref{eq:4bound_square_operator}, the value of $\Lambda$ will be substituted by $M^3$.
\end{proof}

\begin{lemma}
For any $f\in C_b(\R_+)$, $t\geq 0$ it holds that
$$\mathbb{E}\left[ \sup_{s\leq t}|M^{n,f}_s|^2\right] \leq \frac{1}{n}32 \|f\|_\infty^2 M^6t.$$
\end{lemma}
\begin{proof}
The proof is the same one as in Proposition \ref{prop:convergence_martingale} using the bound on the jump kernel $K$.
\end{proof}


\begin{lemma}
\label{eq:limit_trilinear_operator_unbounded_kernel_4}
It holds that for any $t \geq 0$
$$\sup_{s\leq t} |\langle f, Q^{(n)}(X^n_s) - Q(\mu_s)\rangle| \rightarrow 0 \quad \mbox{weakly}$$
for $f$ continuous and of compact support.
\end{lemma}
\begin{proof}
Here everything works as in Section \ref{sec:4convergence_trilinear_term}, but we need to find the bounds \eqref{eq:4bound_difference_trilinear_term} and \eqref{eq:4bound_trilinear_term_uniformly}. We use a similar approach as in \cite{norris1999smoluchowski}.

Firstly, we will prove an analogous bound to \eqref{eq:4bound_trilinear_term_uniformly}. 

Fix $\eps>0$ and define $p(\eps)=\eps^{-1/\gamma}$. Then for $\o\geq p(\eps)$ it holds
$$\frac{\tilde \varphi(\o)}{\o}\leq \eps.$$

Now choose $\kappa\in (0, \gamma/[2(1-\gamma)])$.  We split the domain into $F_1^p:=\{(\o_1,\o_2, \o_3):\, \o_1\leq p^\kappa(\eps),\o_2\leq p^\kappa(\eps), \o_3\leq p^\kappa(\eps)\}$ and $F_2^p$ its complementary. In $F_1^p$ the kernel is bounded and we have, with obvious notations,
$$\sup_{s\leq t} |\langle f, Q_1(X^n_s)-Q_1(\mu_s) \rangle | \rightarrow 0 \quad \mbox{weakly.}$$
On the other hand, in $F_2^p$, at least one of the components is greater than $p^\kappa(\eps)$. Assume, without loss of generality that $\o_3\geq p^\kappa(\eps)$.
Then
\begin{eqnarray*}
|\la f, Q_2(X^n_t) \ra| &=& \bigg|\int_D \left\{f(\om)+f(\o_3)-f(\o_1)-f(\o_2) \right\} \K\\
&& \quad \times X^n_t\doo X^n_t\dotw X^n_t\doth \bigg|\\
&\leq & 4\|f\|_\infty \int_D\tilde\varphi(\o_1)\tilde\varphi(\o_2)\tilde\varphi(\o_3)X^n_t\doo X^n_t\dotw X^n_t\doth\\
&\leq & 4\|f\|_\infty \max\left\{\lp p^{\kappa}(\eps)\rp^{2(1-\gamma)} \eps\la \o, \mu_0\ra, \lp p^\kappa(\eps)\rp^{1-\gamma} \eps^2\la \o, \mu_0\ra^2, \eps^3 \la \o, \mu_0\ra^3 \right\} \\
&\leq& c\eps^\eta \qquad \mbox{for } \eta=1-2\kappa(1-\gamma)/\gamma>0.
\end{eqnarray*}
and analogously
$$|\langle f, Q_2(\mu_t) \rangle |\leq c\eps^\eta.$$
This implies that
$$\limsup_{n\rightarrow \infty} \sup_{s\leq t} |\langle f, Q_2(X^n_s)-Q_2(\mu_s) \rangle | \leq 2c\eps^\eta$$
but $\eps$ is arbitrary so the limit is proved. 

We are left with proving an analogous estimate to \eqref{eq:4bound_difference_trilinear_term}, which is obtained straightforwardly since we restrict ourselves to continuous functions of compact support. 
\end{proof}

\begin{proof}[Proof of Theorem \ref{th:mean_field_limit_unbounded_kernel}]
\label{proof:theorem_mean_limit_1_unbounded}
Thanks to the previous Lemmas we know that there exists  convergent subsequence $X^{n_k}_t \rightarrow \mu_t$ weakly as $k\rightarrow \infty$ such that
$$\langle f, \mu_t \rangle =\langle f, \mu_0\rangle + \int^t_0 \langle f, Q(\mu_s) \rangle ds$$
for any $f$ is continuous of compact support. Now using the bounds on the jump kernel and that $\langle \o, \mu_t\rangle \leq \la \o, \mu_0\ra$ and a limit argument, we can extend this equation to all bounded measurable functions $f$. 
\end{proof}

		\subsection{Second result on mean-field limit}

\subsubsection{A coupling auxiliary process}

Write 
$$X^n_0=\frac{1}{n}\sum^{n}_{i=1}\delta_{\o_i},$$
for $\o_i\in \R_+$. Define for $B\subset \R_+$ bounded
$$X_0^{B,n}=\frac{1}{n} \sum^n_{i\,:\,\o_i\in B} \delta_{\o_i}.$$
Consider $\Lambda^{B,n}_0$ such that for each $B' \subset \R_+$ bounded such that $B\subset B'$ it holds
\begin{equation} \label{eq:bounds_X}
X_0^{B,n} \leq X_0^{B',n}, \quad \la \varphi, X^{B,n}_0 \ra + \Lambda^{B,n}_0 = \la \varphi, X^{B',n}_0\ra + \Lambda^{B',n}_0.
\end{equation}
Set
$$\nu^B = (\Lambda_0^{B,n})^2+2\Lambda_0^{B,n} \la \varphi, X_0^{B,n} \ra - \frac{1}{n^2}\sum_{k,j\,:\, \o_j\notin B \mbox{ or } \o_k\notin B}\varphi(\o_j)\varphi(\o_k) . $$
Note that $\nu^B$ decreases as $B$ increases and $\nu^{[0,\infty)}=(\Lambda_0^{B,n})^2+2\Lambda_0^{B,n} \la \varphi, X_0^{B,n} \ra  \geq 0$.

For $i<j$ take independent exponential random variables $T_{ijk}$ of parameter $K(\o_i, \o_j,\o_k)/n^2$. Set $T_{ijk}=T_{jik}$. Also, for $i\neq j$, take independent exponential random variables $S_{ijk}$ of parameter $\lp \varphi(\o_i)\varphi(\o_j)\varphi(\o_k)-K(\o_i,\o_j,\o_k)\rp/n^2$ (in all these cases we assume that $\o_i+\o_j\geq\o_k$). We can construct, independently for each $i$, a family of independent exponential random variables $S^B_i$, increasing in $B$, with $S^B_i$ having parameter $\varphi(\o_i)\nu^B$.

Set
$$T^B_i=\min_{k,j\,:\, \o_j\notin B \mbox{ or } \o_k\notin B} \lp T_{ijk} \wedge S_{ijk}\rp \wedge S^B_i,$$
$T^B_i$ is an exponential random variable of parameter
$$\frac{1}{n^2}\sum_{k,j\,:\, \o_j\notin B \mbox{ or } \o_k\notin B} \varphi(\o_i)\varphi(\o_j)\varphi(\o_k)+ \varphi(\o_i)\nu^B= \varphi(\o_i)\lp(\Lambda^{B,n}_0)^2+2\Lambda^{B,n}_0 \la \varphi,X_0^{B,n} \ra \rp.$$

For each $B$, the random variables
$$(T_{ijk}, T^B_i:\, i,j,k \mbox{ such that }\o_i, \o_j, \o_k\in B, \, i<j)$$
form an independent family. Suppose that $i$ is such that $\o_i\in B$ and that $j$ is such that $\o_j\notin B$ or $k$ is such that $\o_k\notin B$, then we have
$$T^B_i \leq T_{ijk}$$
and for $B\subset B'$ and all $i$, we have (as a consequence of \eqref{eq:bounds_X})
$$T^B_i \leq T^{B'}_i.$$
Now set
$$T=\lp \min_{i<j,k} T_{ijk} \rp \wedge \lp \min_i T_i^B\rp.$$

We set $(X^{B,n}_t, \Lambda^{B,n}_t) = (X^{B,n}_0, \Lambda^{B,n}_0)$ for $t<T$ and set
$$
(X^{B,n}_T, \Lambda^{B,n}_T) =\left\{ \begin{array}{l}
(X^{B,n}_0 -\frac{1}{n}\delta_{\o_i}-\frac{1}{n}\delta_{\o_j}+\frac{1}{n}\delta_{\o_k}+\frac{1}{n}\delta_{\o_i+\o_j-\o_k}, \Lambda^{B,n}_0)\\
\quad \mbox{if } T=T_{ijk},\,\o_i, \o_j,\o_k,\o_i+\o_j-\o_k\in B,\\
\\
(X^{B,n}_0 -\frac{1}{n}\delta_{\o_i}-\frac{1}{n}\delta_{\o_j}+\frac{1}{n}\delta_{\o_k}, \Lambda^{B,n}_0+\frac{1}{n}\varphi(\o_i+\o_j-\o_k))\\
\quad \mbox{if } T=T_{ijk},\,\o_i, \o_j,\o_k \in B,\;\o_i+\o_j-\o_k\notin B,\\
\\
(X^{B,n}_0- \frac{1}{n}\delta_{\o_i}, \Lambda^{B,n}_0+\frac{1}{n}\varphi(\o_i)), \quad \mbox{if } T=T^B_i,\, \o_i\in B,\\
\\
(X^{B,n}_0, \Lambda^{B,n}_0), \quad\mbox{otherwise}
\end{array}\right.
$$
One can check that $X^{B,n}_T$ is supported on $B$ and for $B\subset B'$
\begin{equation} \label{eq:stochastic_turbulent_bounds}
X^{B,n}_T \leq X^{B',n}_T, \qquad \la \varphi, X^{B,n}_T \ra + \Lambda^B_T = \la \varphi, X^{B',n}_T\ra + \Lambda^{B'}_T .
\end{equation}
We repeat the above construction independently from time $T$, again and again to obtain a family of Markov processes $(X^{B,n}_t, \Lambda^{B,n}_t)_{t\geq0}$ such that \eqref{eq:stochastic_turbulent_bounds} holds for all time.

\begin{remark}
Notice that $\Lambda^{B,n}_0$ and $X^{B,n}_0$ in the definition of $\nu^B$ must be updated to $\Lambda^{B,n}_T$ and $X^{B,n}_T$ in the new step.
\end{remark}

For a bounded set $B\subset [0,\infty)$, we will consider 
$$X^{B,n}_0=\mathbb{1}_B X^n_0, \quad \Lambda^{B,n}_0=\la \varphi \mathbb{1}_{B^c}, X^n_0\ra.$$

\paragraph{Markov Chain generator}
For all $F\in C_b(\mathcal{M}^B)$, $\mu\in \mathcal{M}^B$ we have
\begin{eqnarray*}
\mathcal{G}F(\mu,\lambda) &=& \frac{n}{2}\int_D \left\{ F\lp  \mu^{\o_1,\o_2,\o_3}, \lambda\rp - F(\mu, \lambda) \right\} \mathbb{1}_{\om\in B} \K \mu^{(n)}(d\o_1,d\o_2,d\o_3)\\
&+& \frac{n}{2}\int_D \left\{  F\lp  \hat\mu^{\o_1,\o_2,\o_3}, \lambda^{\om}\rp - F(\mu, \lambda)\right\} \mathbb{1}_{\om\notin B} \K\mu^{(n)}(d\o_1,d\o_2,d\o_3)\\
&+&n \int_{\R_+} \left\{ F\lp  \mu^{\omega}, \lambda^\omega\rp - F(\mu, \lambda)\right\}\lp \lambda^2 + 2\lambda\la \varphi, \mu \ra \rp \varphi(\omega)\mu(d\omega)
\end{eqnarray*}
where
\begin{eqnarray*}
\mu^{\o_1,\o_2,\o_3}&=&\mu+\frac{1}{n}\lp\delta_{\o_3}+\delta_{\om}-\delta_{\o_1}-\delta_{\o_2} \rp;\\
\hat \mu^{\o_1,\o_2,\o_3} &=& \mu + \frac{1}{n}\lp\delta_{\o_3}-\delta_{\o_1}-\delta_{\o_2} \rp;\\
\lambda^{\om} &=& \lambda + \frac{1}{n}\varphi(\om);\\
\lambda^\omega &=& \lambda+\frac{1}{n}\varphi(\omega);\\
\mu^\omega &=& \mu -\frac{1}{n}\delta_\omega
\end{eqnarray*}

\paragraph{Associated martingale.}
Remember the definition
$$\mu^{(n)}(A\times B\times C)=\mu(A)\mu(B)\mu(C)-n^{-1}\mu(A\cap B)\mu(C)$$
which has the property $n^{3}\mu^{(n)}=(n\mu)^{(1)}$. Define for any bounded measurable function $f$ on $\R_+$ and $a\in \R$:
\begin{eqnarray*}
L^{B,(n)}(\mu,\lambda)(f,a) &=& \la (f,a), L^{B,(n)}(\mu,\lambda)\ra\\
&=& \frac{1}{2}\int_{\R_+^3} \big( f(\om)\mathbb{1}_{\om\in B}+a \varphi(\om)\mathbb{1}_{\om\notin B} \\
&&\qquad+ f(\o_3)-f(\o_1)-f(\o_2) \big) \K \mu^{(n)}(d\o_1,d\o_2,d\o_3)\\
&&+\lp \lambda^2+2\lambda \la \varphi, \mu\ra \rp\int_{\R_+}\lp a\varphi(\o)-f(\o) \rp \varphi(\o)\mu(d\o)
\end{eqnarray*}
and
\begin{eqnarray*}
P^{B,(n)}(\mu,\lambda)(f,a)
&=& \frac{1}{2n}\int_{D} \Big( f(\om)\mathbb{1}_{\om\in B}+a \varphi(\om)\mathbb{1}_{\om\notin B} \\
&&\qquad+ f(\o_3)-f(\o_1)-f(\o_2) \Big)^2 \K \mu^{(n)}(d\o_1,d\o_2,d\o_3)\\
&&+\lp \lambda^2+2\lambda \la \varphi, \mu\ra \rp \int_{\R_+}\lp a\varphi(\o)-f(\o) \rp^2 \varphi(\o)\mu(d\o).
\end{eqnarray*}
Then, for all $f$ and $a$
$$M^n_t = \la f, X^{B,n}_t\ra +a\Lambda^{B,n}_t - \la f, X^{B,n}_0\ra - a\Lambda^{B,n}_0 - \int^t_0 L^{B,(n)}(X^{B,n}_s, \Lambda^{B,n}_s)(f,a)\, ds$$
is a martingale with previsible increasing process
$$\la M\ra_t =\int^t_0 P^{B,(n)}(X^{B,n}_s, \Lambda^{B,n}_s)(f,a)\, ds.$$

\subsubsection{Proof of Theorem \ref{th:mean-field-limit-complete}}

Remember the metric $d$ in $\mathcal{M}^f$ defined around expression \eqref{eq:definition_metric_weak_convergence}.

\begin{proposition}\label{prop:comparison_auxiliary_problem_stochastic}
Let $B\subset [0,\infty)$ be bounded and $\mu_0$ be measure on $\R_+$ such that $\la\varphi,\mu_0\ra<\infty$ and that
$$\mu^{*n}_0(\partial B) =0 \quad \mbox{for all } n\geq 1.$$
Assume that for $\varphi(\o)=\o+1$ it holds
$$\K \leq \phiall.$$ 
Consider $(\mu^B_t, \lambda^B_t)_{t\geq 0}$ the solution to \eqref{eq:auxiliary_equation_existence_proof} given by Proposition \ref{prop:existence_auxiliary_process}. Suppose that
$$d(X^{B,n}_0, \mu^{B}_0 )\to 0, \quad
|\Lambda^{B,n}_0-\lambda_0^B|\to 0$$
as $n\to \infty$. Then for all $t \geq 0$,
$$\sup_{s\leq t} d(X^{B,n}_s, \mu^B_s) \to 0,\quad \sup_{s\leq t}|\Lambda^{B,n}_s - \lambda^B_s| \to 0$$
in probability.
\end{proposition}

\begin{proof}[Proof of Proposition \ref{prop:comparison_auxiliary_problem_stochastic}]
Set $M=\sup_{n}\la \varphi, X^{B,n}_0\ra <\infty$. 
 For all B and all continuous bounded functions $f$ and all $a\in \R$
\begin{eqnarray} \label{eq:martingale2}
M^n_t &=& \la f, X^{B,n}_t\ra + a \Lambda^{B,n}_t - \la f, X^{B,n}_0\ra -a \Lambda^{B,n}_0\\ \nonumber
&& \qquad - \int^t_0 L^{B,(n)}(X^{B,n}_s, \Lambda^{B,n}_s)(f,a)\, ds
\end{eqnarray}
is a martingale with previsible increasing process
$$\la M^n\ra_t = \int^t_0 P^{B,(n)}(X^{B,n}_s, \Lambda^{B,n}_s) (f,a)\, ds,$$
(which is the analogous expression to \eqref{eq:alpha_previsible}). 

There is a constant $C<\infty$, depending only on $B, \Lambda, \varphi$, such that
\begin{eqnarray} \label{eq:ineq1}
|L^B(X^{B,n}_t, \Lambda^{B,n}_t)(f,a)| &\leq& C(\|f\|_\infty+|a|)\\
|(L^B-L^{B,(n)})(X^{B,n}_t, \Lambda^{B,n}_t)(f,a)| & \leq & C n^{-1} (\|f\|_\infty+|a|),\\
|P^{B,(n)}(X^{B,n}_t, \Lambda_t^{B,n})(f,a)| &\leq & Cn^{-1}(\|f\|_\infty+|a|)^2, \label{eq:ineq3}
\end{eqnarray}
where $L^B$ is defined in expression \eqref{eq:auxiliary_equation_existence_proof}.

Hence by the same argument as in Theorem \ref{th:hydrodynamic_limit}, the laws of the sequence $(X^{B,n}, \Lambda^{B,n})$ are tight in $D([0,\infty), \mathcal{M}_B\times \R)$ (inequality \eqref{eq:ineq3} is the analogous to \eqref{eq:alpha_previsible}; the inequality \eqref{eq:ineq1} is analogous to \eqref{eq:4bound_square_operator}).

Similarly, the laws of the sequence $(X^{B,n}, \Lambda^{B,n}, I^n, J^n)$ are tight in $D([0,\infty), \mathcal{M}_B \times \R\times \mathcal{M}_{B\times B\times B}\times \mathcal{M}_{B\times B\times B})$, where
\begin{eqnarray*}
I^n_t(d\o_1,d\o_2,d\o_3) &=& \K \mathbb{1}_{\om\in B} X^{B,n}_t(d\o_1) X^{B,n}_t(d\o_2) X^{B,n}_t(d\o_3),\\
J^n_t(d\o_1,d\o_2,d\o_3) &=& \K \mathbb{1}_{\om\notin B} X^{B,n}_t(d\o_1) X^{B,n}_t(d\o_2) X^{B,n}_t(d\o_3).
\end{eqnarray*}

Let $(X,\Lambda, I, J)$ some weak limit point of the sequence. Passing to a subsequence and using the Skorokhod representation theorem \ref{th:Skorokhod_representation_theorem}, we can consider that the sequence converges almost surely, i.e., as a pointwise limit in $D([0,\infty), \mathcal{M}_B \times \R\times \mathcal{M}_{B\times B\times B}\times \mathcal{M}_{B\times B\times B})$. Therefore, there exist bounded measurable functions 
$$I, J:[0,\infty)\times B\times B\times B \to [0,\infty)$$
symmetric in the first two components, such that
\begin{eqnarray*}
I_t(d\o_1, d\o_2, d\o_3)  &=& I(t, \o_1, \o_2, \o_3) X_t(d\o_1) X_t(d\o_2) X_t(d\o_3)\\
J_t(d\o_1, d\o_2, d\o_3)  &=& J(t, \o_1, \o_2, \o_3) X_t(d\o_1) X_t(d\o_2) X_t(d\o_3)
\end{eqnarray*}
in $\mathcal{M}_{B\times B \times B}$ and such that
\begin{eqnarray*}
I(t,\o_1,\o_2,\o_3) &=& \K \mathbb{1}_{\om \in B}\\
J(t,\o_1,\o_2,\o_3) &=& \K \mathbb{1}_{\om \notin B}
\end{eqnarray*}
whenever $\om\notin \partial B$ (notice that we assumed $K$ to be continuous). 

Now, passing to the limit in \eqref{eq:martingale2} we obtain, for all continuous functions $f$ and all $a\in \R$, for all $t\geq 0$, almost surely
\begin{eqnarray*}
\la (f,a), (X_t, \Lambda_t)\ra &=&\la (f,a), (X_0,\Lambda_0) \ra\\
&&+ \frac{1}{2} \int^t_0 \int_{\R^3_+} \big( f(\om)+f(\o_3)-f(\o_1)-f(\o_2) \big)\\
&& \qquad\times I(s, \o_1, \o_2,\o_3)  X_s(d\o_1) X_s(d\o_2) X_s(d\o_3)\, ds\\
&&+ \frac{1}{2}\int^t_0 \int_{\R^3_+} \lp a\varphi(\om)+f(\o_3)-f(\o_1)-f(\o_2) \rp\\
&& \qquad \times J(s, \o_1, \o_2, \o_3)  X_s(d\o_1)  X_s(d\o_2)  X_s(d\o_3)\, ds\\
&& + \int^t_0 \lp \Lambda_s^2+2\Lambda_s\la \varphi, X_s\ra  \rp\int_{\R_+} \lp a\varphi(\o)-f(\o) \rp \varphi(\o) X_s(d\o)\, ds.
\end{eqnarray*}

\medskip

Consider now an analogous iterative scheme to the one done in Proposition \ref{prop:existence_auxiliary_process} for this equation. Denote by $(\nu_t^n)_{n\in\NN}$ the sequence approximating $(X_t)_{t\geq 0}$. We deduce that
$$\nu^0_t=\mu_0, \quad\nu^{n+1}_t \ll \mu_0 + \int^t_0 (\nu^n_s + \nu^n_s * \nu^n_s *\hat\nu^n_s) \, ds$$
for $\hat\nu(A)=\nu(-A)$ and for all $n \geq 0$, (notice that we have extended the measures in the previous expression to the whole $\R$ by taking value 0 in subsets of $(-\infty, 0)$)\footnote{$$\la f, \nu * \nu *\hat\nu \ra = \int_{\R^3} f(\om) \nu(d\o_1)\nu(d\o_2)\nu(d\o_3)$$ }.

 By induction we have that
$$\nu^n_t \ll \gamma_0=\sum^\infty_{k=1}\sum^\infty_{l=0}\nu_0^{*k}*\hat\nu_0^{*l}.$$
This implies in our case (taking $n\to \infty$) that $X_t \otimes X_t \otimes X_t$ is absolutely continuous with respect to $\gamma_0^{\otimes 3}$
for all $t\geq 0$, almost surely. For $G=\{(\o_1,\o_2,\o_3)\, |\, \om\in \partial B\}$, we have that $\gamma_0^{\otimes 3}(G)=0$ because of the assumptions on $\mu_0$ and that $\gamma_0^{\otimes 3}(G) = (\gamma_0 * \gamma_0 *\hat{\gamma_0})(G)$.

Therefore we can replace $I(t, \o_1, \o_2,\o_3)$ by $\K\mathbb{1}_{\om \in B}$ and $J(t,\o_1,\o_2,\o_3)$ by $\K \mathbb{1}_{\om\notin B}$. Since the equation obtained after this substitution is the same as \eqref{eq:auxiliary_equation_existence_proof} and $(\mu^B_t, \lambda^B_t)$ is its unique solution, we conclude that the unique weak limit point of $(X^{B,n},\Lambda^{B,n})$ in $D([0,\infty), \mathcal{M}_B \times \R)$ is precisely $(\mu^B_t, \lambda^B_t)_{t\geq 0}$.

\end{proof}

The proof of Theorem \ref{th:mean-field-limit-complete} is exactly the same one as in \cite[Theorem 4.4]{norris1999smoluchowski} and we copy it here just for the sake of completeness.
\begin{proof}[Proof of Theorem \ref{th:mean-field-limit-complete}, from \cite{norris1999smoluchowski}]
Fix $\delta>0$ and $t<T$. Since $(\mu_t)_{t<T}$ is strong, we can find a compact set $B$ satisfying $\mu^{*n}_0(\partial B)=0$(\footnote{The reason for this being true is that, for any given $\mu_0$, $\mu^{*n}_0(\partial B)=0$ for all $n\geq 1$ holds for all but countably many closed intervals in $\R_+$.}) for all $n\geq 1$ and such that $\lambda^B_t<\delta/2$. Now
$$d(\varphi X^n_0, \varphi \mu_0) \to 0,$$
so
$$d(X^{B,n}_0,\mu^B_0)\to 0, \quad |\Lambda^{B,n}_0 - \lambda^B_0|\to 0.$$
Hence, by Proposition \ref{prop:comparison_auxiliary_problem_stochastic},
$$\sup_{s\leq t} d(X^{B,n}_s, \mu^B_s) \to 0, \quad \sup_{s\leq t}|\Lambda^{B,n}_s-\lambda^B_s|\to 0,$$
in probability as $n\to \infty$. Since $\{\mu^{B}_s:\, s\leq t\}$ is compact (the support of $\mu_s$ is contained in $B$)(\footnote{Remember the definition of the metric $d$ given in \eqref{eq:definition_metric_weak_convergence}. Since $d(X^{B,n}_s, \mu^B_s) \to 0$, we have that for all $f$ bounded continuous function on $\R_+$
$$\int f \varphi(X^{B,n}_s-\mu^B_s) =\int f \varphi\mathbb{1}_B (X^{B,n}_s-\mu^B_s) \to 0$$
since $\varphi$ restricted to $B$ is also bounded and continuous.
}), we also have 
$$\sup_{s\leq t} d(\varphi X^{B,n}_s, \varphi\mu^{B}_s) \to 0$$
in probability as $n\to \infty$. By \eqref{eq:bounds_compare_sol_aux_sol} and by the bounds on the instantaneous coagulation-fragmentation particle system \eqref{eq:stochastic_turbulent_bounds}, we have that for $s\leq t$
\begin{eqnarray*}
\|\varphi(\mu_s-\mu_s^B) \| &=& \la \varphi, \mu_s-\mu_s^B \ra \leq \lambda^B_s \leq \lambda^B_t <\delta/2\\
\|\varphi(X^n_s-X^{B,n}_s)\| &=& \la \varphi, X^n_s-X^{B,n}_s \ra \leq \Lambda^{B,n}_s \leq \Lambda^{B,n}_t \\
&\leq & \lambda^B_t + |\Lambda^{B,n}_t-\lambda^B_t|\\
&\leq & \delta/2 + |\Lambda^{B,n}_t - \lambda^B_t|.
\end{eqnarray*}
Now (remember the properties of the metric $d$ defined in \eqref{eq:definition_metric_weak_convergence})
\begin{eqnarray*}
d(\varphi X^n_s, \varphi \mu_s) &\leq & \|\varphi(X^n_s - X^{B,n}_s)\| + d(\varphi X^{B,n}_s, \varphi \mu^B_s) + \|\varphi(\mu_s-\mu^B_s)\|\\
&\leq & \delta + d(\varphi X^{B,n}_s, \varphi \mu^B_s) + |\Lambda^{B,n}_t -\lambda^B_t|,
\end{eqnarray*}
so
$$\mathbb{P}\lp\sup_{s\leq t} d(\varphi X^n_s, \varphi \mu_s)>\delta \rp \to 0
$$
as $n\to\infty$, as required.
\end{proof}


\section{Appendix: Some properties of the Skorokhod space}
\begin{theorem}[Prohorov's theorem (\cite{ethier2009markov}), Chapter 3]
Let $(S,d)$ be complete and separable, and let $\mathcal{M} \in \P(S)$. Then the following are equivalent:
\begin{enumerate}
	\item $\mathcal{M}$ is tight.
	\item For each $\eps>0$, there exists a compact $K\in S$ such that
	$$\inf_{P\in \mathcal{M}}P(K^\eps) \geq 1-\eps$$
	where $K^\eps:= \{x\in S: \inf_{y\in K} d(x,y)<\eps\}$.
	\item $\mathcal{M}$ is relatively compact.
\end{enumerate}
\end{theorem}

Let $(E,r)$ be a metric space. The space $D([0,\infty); E)$ of cadlag functions taking values in $E$ is widely used in stochastic processes. In general we would like to study the convergence of measures on this space, however, most of the tools known for convergence of measures are for measures in $\P(S)$ for $S$ a complete separable metric space. Therefore, it would be very useful to find a topology in $D([0,\infty) \times E)$ such that it is a complete and separable metric space. This can be done when $E$ is also complete and separable; and the metric considered is the Skorokhod one. This is why in this case the space of c\`adl\`ag functions is called Skorohod space. 

Some important properties of this space are the following:
\begin{proposition}[\cite{ethier2009markov}, Chapter 3]
If $x\in D([0,\infty); E)$, then $x$ has at most countably many points of discontinuity.
\end{proposition}

\begin{theorem}[\cite{ethier2009markov}, Chapter 3]
If $E$ is separable, then $D([0,\infty); E)$ is separable. If $(E,r)$ is complete, then $(D([0,\infty);E), d)$ is complete, where $d$ is the Skorokhod metric. 

\end{theorem}

\begin{theorem}
The Skorokhod space is a complete separable metric space.
\end{theorem}

\begin{theorem}[The almost sure Skorokhod representation theorem, \cite{ethier2009markov}, Theorem 1.8, Chapter 3]
\label{th:Skorokhod_representation_theorem}
Let $(S,d)$ be a separable metric space. Suppose $P_n$, $n=1,2,\hdots$ and $P$  in $\P(S)$ satisfy $\lim_{n\rightarrow\infty}\rho(P_n, P)=0$ where $\rho$ is the metric in $\P(S)$. Then there exists a probability space $(\Omega, \mathcal{F}, \nu)$ on which are defined $S$- valued random variable $X_n$, $n=1,2, \hdots$ and $X$ with distributions $P_n$, $n=1,2,\hdots$ and $P$, respectively such that $\lim_{n\rightarrow \infty} X_n=X$ almost surely.
\end{theorem}

\begin{theorem}[Tightness criteria for measures on the Skorokhod space, \cite{jakubowski1986skorokhod} Theorem 3.1]
\label{th:jakubowski_criteria}
Let $(S,\mathcal{T})$ be a completely regular topological space with metrisable compact sets. Let $\mathbb{G}$ be a family of continuous functions on $S$. Suppose that $\mathbb{G}$ separates points in $S$ and that it is closed under addition. Then a family $\{ \mathcal{L}^n\}_{n\in \NN}$ of probability measures in $\P(D([0,\infty);S)$ is tight iff  the two following conditions hold:
\begin{itemize}
	\item[(i)] For each $\eps>0$ there is a compact set $K_\eps \subset S$ such that
	$$\mathcal{L}^n(D([0,\infty);K_\eps))>1-\eps, \quad n\in \NN.$$
	\item[(ii)] The family $\{ \mathcal{L}^n\}_{n\in\NN}$ is $\mathbb{G}$-weakly tight.
\end{itemize}
\end{theorem}

\begin{theorem}[Criteria for tightness in Skorokhod spaces (\cite{ethier2009markov}, Corollary 7.4, Chapter 3)]
\label{th:criteria_tightness_Skorokhod}
Let $(E,r)$ be a complete and separable metric space, and let $\{X_n\}$ be a family of processes with sample paths in $D([0,\infty); E)$. Then $\{X_n\}$ is relatively compact iff the two following conditions hold:
\begin{itemize}
	\item[(i)]For every $\eta>0$ and rational $t\geq 0$, there exists a compact set $\Lambda_{\eta, t} \subset E$ such that
	$$\liminf_{n\rightarrow \infty}\mathbb{P}\{X_n(t) \in \Lambda^\eta_{\eta,t}\} \geq 1-\eta.$$
	\item[(ii)] For every $\eta>0$ and $T>0$, there exits $\delta>0$ such that
	$$\limsup_{n\rightarrow\infty} \mathbb{P}\{ w'(X_n, \delta, T) \geq \eta \} \leq \eta.$$
\end{itemize}
where we have used the \textbf{modulus of continuity} $w'$ defined as follows: for $x\in D([0,\infty)\times E)$, $\delta>0$, and $T>0$:
$$w'(x,\delta, T) = \inf_{\{t_i\}} \max_i \sup_{s,t \in [t_{i-1}, t_i)} r(x(s), x(t)),$$
where $\{t_i\}$ ranges over all partitions of the form $0=t_0<t_1<\hdots< t_{n-1}<T\leq t_n$ with $\min_{1\leq i\leq n}(t_i-t_{i-1})>\delta$ and $n\geq 1$
\end{theorem}

\begin{theorem}[Continuity criteria for the limit in Skorokhod spaces (\cite{ethier2009markov}, Theorem 10.2, Chapter 3)]
\label{th:continuity_criteria_limit_Skorokhod_space}
Let $(E,r)$ be a metric space.
Let $X_n$, $n=1,2,\hdots,$ and $X$ be processes with sample paths in $D([0,\infty);E)$ and suppose that $X_n$ converges in distribution to $X$. Then $X$ is a.s. continuous if and only if $J(X_n)$ converges to zero in distribution, where
$$J(x) = \int^\infty_0 e^{-u} [ J(x,u) \wedge 1] \, du$$
for 
$$J(x,u) = \sup_{0\leq t \leq u} r(x(t), x(t-)).$$
\end{theorem}

	\section{Appendix: Formal derivation of the weak isotropic 4-wave kinetic equation}
		\label{sec:weak_isotropic_wave_eq}


Suppose that $n(\kvec)=n(k)$ is a radial function (isotropic).

The waveaction in the isotropic case can be written as
$$W= \int_{\R^N} n(\kvec) d\kvec = \int_{\R_+\times S^{N-1}} n(k) k^{N-1} dk d\s = \frac{|S^{N-1}|}{\alpha} \int^\infty_0 n(\omega) \o^{\frac{N-\alpha}{\alpha}} \, d\o ,$$
 where $S^{N-1}$ is the $N-1$ dimensional sphere. From this expression, one can denote the angle-averaged frequency spectrum $\mu=\mu(d\o)$ as 
 $$\mu(d\omega):=\frac{|S^{N-1}|}{\alpha} \o^{\frac{N-\alpha}{\alpha}}n(\o)d\o.$$
The total number of waves (waveaction) and the total energy are respectively
\begin{eqnarray*} 
W&=& \int^\infty_0 \mu(d\omega)\\
E &=& \int^\infty_0 \omega \mu(d\omega).
\end{eqnarray*}

The isotropic version of the weak 4-wave kinetic equation can be written as
\begin{equation} \label{eq:isotropic_4_wave_equation_weak2}
\mu_t = \mu_0 + \int^t_0 Q(\mu_s, \mu_s, \mu_s) \, ds
\end{equation}
where $Q$ is defined against test functions $g\in \mathcal{S}(\R_+)$ as
\begin{eqnarray} \label{eq:Q_isotropic_case}
\la g, Q(\mu, \mu,\mu) \ra &=& \frac{1}{2} \int_{D}  \mu\doo \mu\dotw \mu\doth K(\o_1,\o_2,\o_3)  \\ \nonumber
&&\qquad \times[ g(\o_1+\o_2-\o_3) + g(\o_3) -g(\o_2) -g(\o_1) ]
\end{eqnarray}
where $D:= \{ \R^3_+ \cap (\o_1 + \o_2 \geq \o_3)\}$ and
\begin{eqnarray} \label{eq:jump_kernel_isotropic_4wave}
\K &=& \frac{8\pi}{\alpha |S^{N-1}|^4}(\om)^{\frac{N-\alpha}{\alpha}}\\
&& \qquad \int_{\lp S^{N-1}\rp^4} d\s_1d\s_2d\s_3d\s\, \overline{T}^2(\o_1^{1/\alpha} \s_1, \o_2^{1/\alpha}\s_2, \o_3^{1/\alpha}\s_3, (\om)^{1/\alpha} \s) \nonumber\\
&& \qquad \qquad \times \delta(\o_1^{1/\alpha}\s_1+ \o_2^{1/\alpha}\s_2- \o_3^{1/\alpha}\s_3- (\om)^{1/\alpha} \s) \nonumber
\end{eqnarray}

\bigskip

Next we explain the formal derivation of the weak isotropic 4-wave kinetic equation \eqref{eq:isotropic_4_wave_equation_weak}. We have that 
\begin{eqnarray*}
\int_{(0,\infty)} \partial_t \mu(\omega) d\o &=& \int_{\R^N} \partial_t n(\kvec) d\kvec\\
&=& 4\pi \int_{\Omega^4 \times S^4} \overline{T}^2(k_1 s_1, k_2 s_2, k_3 s_3, k s)\\
&&\qquad\qquad\times \delta(k_1 s_1+k_2 s_2-k_3s_3-ks) \delta(\omega_1+\o_2-\o_3-\o)\\
&& \qquad\qquad\times (n_1n_2n_3+n_1n_2n-n_1n_3n-n_2n_3n) (k k_1k_2k_3)^{N-1}dkds\\
&=&\frac{4\pi}{\alpha|S^{N-1}|^4} \int_{\R_+^4 \times S^4} d\o_{0123} ds_{0123}T^2(\o_1^{1/\alpha} s_1, \o_2^{1/\alpha} s_2, \o^{1/\alpha}_3 s_3, \o^{1/\alpha} s)\\
&&\qquad\qquad\times \delta(\o^{1/\alpha}_1 s_1+\o^{1/\alpha}_2 s_2-\o^{1/\alpha}_3s_3-\o^{1/\alpha}s) \delta(\omega_1+\o_2-\o_3-\o)\\
&& \qquad\qquad\times (\mu(\o_1)\mu(\o_2)\mu(\o_3)\o^{\frac{N-\alpha}{\alpha}}+\mu(\o_1)\mu(\o_2)\mu(\o) \o_3^{\frac{N-\alpha}{\alpha}}\\
&&\qquad\qquad\quad-\mu(\o_1)\mu(\o_3)\mu(\o)\o_2^{\frac{N-\alpha}{\alpha}}-\mu(\o_2)\mu(\o_3)\mu(\o)\o_1^{\frac{N-\alpha}{\alpha}})  \\
&=&\int_{\R_+^4}  d\o_{0123} F(\o_1,\o_2,\o_3,\o)\delta(\omega_1+\o_2-\o_3-\o)\\
&& \qquad\qquad\times (\mu(\o_1)\mu(\o_2)\mu(\o_3)\o^{\frac{N-\alpha}{\alpha}}+\mu(\o_1)\mu(\o_2)\mu(\o) \o_3^{\frac{N-\alpha}{\alpha}}\\
&& \qquad\qquad\quad-\mu(\o_1)\mu(\o_3)\mu(\o)\o_2^{\frac{N-\alpha}{\alpha}}-\mu(\o_2)\mu(\o_3)\mu(\o)\o_1^{\frac{N-\alpha}{\alpha}})  \\
\end{eqnarray*}
for $S^i=(S^{N-1})^i$, $d\o_{0123}=d\o d\o_1 d\o_2 d\o_3$, $ds_{0123}=ds_1 ds_2 ds_3 ds$, and
\begin{eqnarray*}
F(\o_1, \o_2, \o_3,\o) &=& \frac{4\pi}{\alpha|S^{N-1}|^4}  \int_{S^4} ds_{0123} \overline{T}^2(\o_1^{1/\alpha} s_1, \o_2^{1/\alpha} s_2, \o^{1/\alpha}_3 s_3, \o^{1/\alpha} s)\\
&&\qquad\qquad\times\delta(\o^{1/\alpha}_1 s_1+\o^{1/\alpha}_2 s_2-\o^{1/\alpha}_3s_3-\o^{1/\alpha}s).
\end{eqnarray*}

\medskip
Hence, $\mu_\omega$ satisfies
\begin{eqnarray}\label{eq:isotropic_4_waveKineticEquation}
\partial_t \mu(\o) &=& \int_{\R_+^3}  d\o_{123} F(\o_1,\o_2,\o_3,\o)\delta(\omega_1+\o_2-\o_3-\o)\\
&& \qquad\qquad\times (\mu(\o_1)\mu(\o_2)\mu(\o_3)\o^{\frac{N-\alpha}{\alpha}}+\mu(\o_1)\mu(\o_2)\mu(\o) \o_3^{\frac{N-\alpha}{\alpha}}\\
&& \qquad\qquad-\mu(\o_1)\mu(\o_3)\mu(\o)\o_2^{\frac{N-\alpha}{\alpha}}-\mu(\o_2)\mu(\o_3)\mu(\o)\o_1^{\frac{N-\alpha}{\alpha}}) \nonumber
\end{eqnarray}

Its weak formulation
$$\mu_t =\mu^{in} +\int_{\Omega^3} Q(\mu_s, \mu_s, \mu_s) \, ds$$
is defined against functions $g\in \mathcal{S}(\R_+)$ as
\begin{eqnarray} \nonumber
\langle g, Q(\mu, \mu,\mu) \rangle &=& \int_{\R_+^4} d\o_{0123} \mu(\o_1) \mu(\o_2) \mu(\o_3) \o^{\frac{N-\alpha}{\alpha}} \\ \nonumber
&&\qquad \times[ F_{1230} \delta(\o^{12}_{30}) g(\o) + F_{1203} \delta(\o^{12}_{03})g(\o_3) \\ \nonumber
&& \qquad\quad - F_{1032}\delta(\o^{10}_{32})g(\o_2) -F_{0231}\delta(\o^{02}_{31})g(\o_1) ]\\ \nonumber
&=& \int_{\R_+^4} d\o_{0123} \mu(\o_1) \mu(\o_2) \mu(\o_3) \o^{\frac{N-\alpha}{\alpha}} F_{1230} \delta(\o^{12}_{30})  \\ \nonumber
&&\qquad \times[ g(\o) + g(\o_3) -g(\o_2) -g(\o_1) ]\\ \nonumber
&=& \frac{1}{2}\int_{D} d\o_{123} \mu(\o_1) \mu(\o_2) \mu(\o_3)  K(\o_1,\o_2,\o_3)  \\
&&\qquad \times[ g(\o_1+\o_2-\o_3) + g(\o_3) -g(\o_2) -g(\o_1) ]
\end{eqnarray}
To conclude we assumed that $\overline{T}$ is symmetric in all its variables. We used that changing labels we get that 
\begin{eqnarray*}
&&d\o_{123} F_{1230} \delta(\o^{12}_{30}) g(\o)+ F_{1203} \delta(\o^{12}_{03}) g(\o_3)- F_{1032}\delta(\o^{10}_{32})g(\o_2) -F_{0231}\delta(\o^{02}_{31})g(\o_1)\\ && \qquad=d\o_{123} F_{1230} \delta(\o^{12}_{30}) g(\o)+ F_{1203} \delta(\o^{12}_{03}) g(\o_3)- F_{3012}\delta(\o^{30}_{12})g(\o_2) -F_{0321}\delta(\o^{03}_{21})g(\o_1)
\end{eqnarray*}
 and the properties of the function $F$ to factorise it. We used the notation $\delta(\o^{ij}_{lp})=\delta(\o_i+\o_j-\o_l-\o_p)$ and
$$K(\o_1, \o_2, \o_3):= 2(\o_1 +\o_2-\o_3)^{\frac{N-\alpha}{\alpha}}F(\o_1, \o_2,\o_3, \o_1+\o_2-\o_3).$$
For the last line we used the \textit{sifting property} of the delta distribution i.e.
\begin{equation} \label{eq:sifting_property}
\int^b_a f(t) \delta(t-d) \, dt = \left\{\begin{array}{l l}
f(d) & \mbox{for } d\in[a, b]\\
0 & \mbox{otherwise}
\end{array} \right. .
\end{equation}

\begin{remark}
In reference \cite[Section 3.1.3]{zakharov1992kolmogorov}, the authors state that even in isotropic medium, the interaction coefficient $\overline{T}$ in the 4-wave case cannot be considered to be isotropic too. In the 3-wave case it is possible, but not for the 4-wave. We can rewrite
\begin{equation} \label{eq:def_f2}
|\overline{T}(\kvec_1, \kvec_2, \kvec_3, \kvec)|^2 = \overline{T}_0^2 k^{2\beta} f_2\lp\frac{\kvec_1}{k},\frac{\kvec_2}{k},\frac{\kvec_3}{k}\rp
\end{equation}
for some dimensionless constant $\overline{T}_0$ and some dimensionless function $f_2$. 
\end{remark}

			\section{Conclusions}
	
In this work we have dealt with the weak isotropic 4-wave kinetic equation with simplified kernels. When the kernels are at most linear we have given conditions for the local existence and uniqueness of solutions. We have also derived the equation as a mean-field limit of interacting particle system given by a simultaneous coagulation-fragmentation: three particles interact with a coagulation-fragmentation phenomenon where one of the particles seem to act as a catalyst. 	

As we saw in the introduction, this theory can be applied to physical scenarios that include Langmuir waves, shallow water and waves on elastic plates. Moreover, using the interacting particle system, numerical methods could be devised to simulate the solution of the equation (as done by \cite{connaughton2009numerical} for the 3-wave kinetic equation), by adapting the methods in \cite{eibeck2000approximative}.

Finally, these numerical simulations would allow the study of steady state solutions and to check if they match the Kolmogorov-Zakharov spectra. This will be attempted in a future work.
	
	\bibliographystyle{amsalpha}
\bibliography{biblio}

\end{document}